\documentclass[12pt]{amsart}
\usepackage{amsmath,xcolor,mathrsfs}
\usepackage{amssymb}
\usepackage[shortlabels]{enumitem}
\usepackage{verbatim}
\input xy
\xyoption{all}

\usepackage{graphicx}

\usepackage[cmtip,all]{xy}

\usepackage{}

\newtheorem{theorem}{Theorem}[section]
\newtheorem{lemma}[theorem]{Lemma}

\newtheorem{definition}[theorem]{Definition}
\newtheorem{Lemma}[theorem]{Lemma}
\newtheorem{corollary}[theorem]{Corollary}
\newtheorem{Prop}[theorem]{Proposition}

\theoremstyle{definition}
\newtheorem{remark}[theorem]{Remark}
\newtheorem{example}[theorem]{Example}
\numberwithin{equation}{section}
\newcommand{\Cal}[1]{{\mathcal #1}}

\newcommand{\Max}{\operatorname{Max}}
\newcommand{\Spec}{\operatorname{Spec}}

\DeclareMathOperator{\Over}{Over}
\DeclareMathOperator{\Zar}{Zar}
\DeclareMathOperator{\SStar}{SStar}

\newcommand{\cmat}{\left(\begin{array}}
	\newcommand{\fmat}{\end{array}\right)}

\newcommand{\f}{\mathfrak}

\newcommand{\chius}{\mathrm{Cl}}
\newcommand{\cons}{\mathrm{cons}}
\newcommand{\inverse}{\mathrm{inv}}

\newcommand{\xcal}{{\boldsymbol{\mathcal{X}}}}
\newcommand{\overic}{I}
\newcommand{\pruf}{\mathrm{Pruf}}
\newcommand{\prufsloc}{\pruf_{\mathrm{sloc}}}
\newcommand{\insid}{\mathcal{I}}
\newcommand{\primary}{\mathcal{P}}
\newcommand{\inssemistar}{\mathrm{SStar}}
\newcommand{\inssemistartf}{\mathrm{SStar}_f}
\newcommand{\inssubmod}{\mathbf{F}}
\newcommand{\inssubfg}{\mathbf{f}}
\newcommand{\B}{\mathcal{B}}
\newcommand{\insclos}{\mathrm{Clos}}
\newcommand{\rad}{\mathrm{rad}}
\newcommand{\insprinc}{\mathrm{Princ}}
\newcommand{\fin}{\mathrm{fin}}
\newcommand{\zariski}{\mathrm{zar}}

\begin{document}
\title{Suprema in spectral spaces and the constructible closure}
\author{Carmelo Antonio Finocchiaro}
\address{Dipartimento di Matematica e Informatica, Universit\`a degli Studi di Catania, Viale Andrea Doria 6 - 95125, Catania}
\email{cafinocchiaro@unict.it}
\author{Dario Spirito}
\address{Dipartimento di Matematica e Fisica, Universit\`a degli Studi Roma Tre, Largo San Leonardo Murialdo 1, 00146, Roma}
\email{spirito@mat.uniroma3.it}
\begin{abstract}
Given an arbitrary spectral space $X$, we endow it with its specialization order $\leq$ and we study the interplay between suprema  of subsets of $(X,\leq)$ and the constructible topology. More precisely, we investigate about when the supremum of a set $Y\subseteq X$  exists and belongs to the constructible closure of $Y$. We apply such results to algebraic lattices of sets and to closure operations on them, proving density properties of some distinguished spaces of rings and ideals. Furthermore, we provide topological characterizations of some class of domains in terms of topological properties of their ideals. 
\end{abstract}
\maketitle
\section{Introduction}
A topological space $X$ is said to be a \emph{spectral space} if it is homeomorphic to the spectrum of a (commutative, unitary) ring, endowed with the Zariski topology; as shown by Hochster \cite{ho}, being a spectral space is a topological condition, in the sense that it is possible to define spectral spaces exclusively through topological properties, without mentioning any algebraic notion. His proof relies heavily on the passage from the starting topology to a new topology, the \emph{patch} or \emph{constructible} topology (see Section \ref{sect:construct}), which remains spectral but becomes Hausdorff; this topology has recently been interpreted as the topology of ultrafilter limit points with respect to the open and quasi-compact subsets of the original topology (see \cite{Fi} and \cite{Fo-Lo-2008}). Spectral spaces are related to several other topics, for example Boolean algebras, distributive lattices and domain theory, all of which provide a different context and a different point of view on the underlying topological structure. The monograph \cite{di-sc-tr} provides a recent and in-depth presentation of this subject and of its connections with other areas of mathematics. 

The spectrum of a ring carries a natural partial order, the one induced by set inclusion: such order can also be recovered topologically, as it coincides with the specialization order of the Zariski topology. In this paper, we exploit the interplay between this order and the constructible topology to analyze several different spaces of algebraic interest (arising as sets of modules or of rings), in particular determining several examples of subspaces that are dense or closed in the constructible topology. Many of our results are based on Theorem \ref{supremum}, a simple observation  that is consequence of  \cite[Theorem 4.2.6]{di-sc-tr}: if $X$ is a spectral space and $Y\subseteq X$ is closed by finite suprema, then all subsets of $Y$ have a supremum, which belongs to the constructible closure of $Y$. (An analogous result holds for infima, as can be seen through the use of the inverse topology.) In Section \ref{sect:sup}, we study some variants of this result, introducing a condition on spectral spaces (being \emph{locally with maximum}, Definition \ref{defin:locwithmax}) which allows a greater control on the constructible closure, and connecting it with some order topologies induced the specialization order. 

In Section \ref{sect:alglatt}, we bridge the gap toward the algebraic setting by connecting the concept of \emph{algebraic lattice of sets} inside a set $S$ \cite[7.2.12]{di-sc-tr}, which can be used to model several spaces of substructures, with the notion of finite-type closure operation, and show that the set of these closures can be made into a spectral space, generalizing \cite[Theorem 2.13]{FiSp}.

In the last two sections, we apply these results to spaces of submodules, of overrings and of semistar operations, which provide several natural examples of spectral spaces when endowed with the Zariski or the hull-kernel topology. While known results are usually positive, i.e., they concentrate on spaces which are spectral and/or closed in the constructible topology (see for example \cite[Example 2.2]{olberding-irred}), our method allows to find example of subspaces that are dense with respect to the constructible topology and thus, in particular, are not closed. In Section \ref{sect:submodules}, we concentrate on spaces of ideals and modules: in particular, we analyze finitely generated modules (Proposition \ref{prop:fg-dense}), primary ideals of Noetherian rings (Propositions \ref{prop:primary-noeth} and \ref{prop:primary0}) and valuation domains (Proposition \ref{prop:primary-val}) and principal ideals in Noetherian rings (Proposition \ref{prop:princ-noeth-PIR}) and Krull domains (Theorem \ref{thm:Krull}). In Section \ref{sect:overrings} we study overrings: we show that the set of Noetherian valuation overrings of a Noetherian domain is always dense in the Zariski space (Theorem \ref{DVR-dense}) and that the set of Pr\"ufer overrings of an integral domain $D$ is dense in the set of integrally closed overrings of $D$ (Proposition \ref{semilocal-prufer-dense}). At the end of the paper, we analyze the relationship between the space of overrings and the space of finite-type semistar operations on a domain (Proposition \ref{prop:controes-semistar}).

\section{Preliminaries}
Let $X$ be a topological space. Then, $X$ is said to be a \emph{spectral space} if the following conditions hold:
\begin{itemize}
    \item $X$ is quasi-compact and $T_0$;
    \item $X$ has a basis of quasi-compact sets that is closed by finite intersections;
    \item every closed irreducible subset of $X$ has a generic point (i.e., $X$ is \emph{sober}).
\end{itemize}
By \cite[Theorem 6]{ho}, $X$ is a spectral space if and only if there is a (commutative, unitary) ring $A$ such that $X\simeq\Spec(A)$, where $\Spec(A)$ (the \emph{spectrum} of $A$) is endowed with the Zariski topology.

\subsection{The constructible topology.}\label{sect:construct} Let $X$ be a spectral space. The \emph{constructible} (or \emph{patch}) \emph{topology} on $X$ is the coarsest topology for which all open and quasi-compact subspaces of $X$ are clopen sets. In the following we will denote by $X^\cons$ the space $X$, with the constructible topology, and, for every $Y\subseteq X$, by $\chius^\cons(Y)$ the closure of $Y$ in $X^\cons$. In light of \cite[Theorem 1 and Proposition 4]{ho}, $X^\cons$ is quasi-compact, Hausdorff and totally disconnected, and thus, a fortiori, a zero-dimensional spectral space.

A mapping $f:X\to Y$ of spectral spaces is called a \emph{spectral map} if, for every open and quasi-compact subspace $V$ of $Y$, $f^{-1}(V)$ is open and quasi-compact. In particular, any spectral map is continuous; moreover, if $X$ and $Y$ are endowed with the constructible topology, $f$ becomes continuous and closed (see \cite[pag. 45]{ho}). 

%Let $x\in X$, let $\emptyset\neq Y\subseteq X$ and let $\ms U$ be an ultrafilter on $Y$.  We say that $x$ is \emph{an ultrafilter limit point of $Y$ with respect to $\ms U$} if, for every open and quasi quasi-compact subset $S$ of $X$, we have $x\in S$ if and only if $S\cap Y\in\ms U$. 

%\begin{example}
%When $X:=\Spec(A)$, for some ring $A$, and $Y\subseteq \Spec(A)$, it is easily seen that the ultrafilter limit point of $Y$ with respect to an ultrafilter $\ms U$ on $Y$ is the prime ideal 
%$$
%Y_{\ms U}:=\{a\in A\mid V(a)\cap Y\in\ms U \};
%$$
%see, for instance, \cite[Proposition 2.3(1)]{fifolo2}.
%\end{example}

%By \cite[Proposition 2.13 and Proposition 3.2]{Fi}, the closure of $Y$ in $X^\cons$ is exactly the set of all ultrafilter limit points of $Y$ with respect to some ultrafilter $\ms U$ on $Y$.

A \emph{proconstructible} subset of $X$ is a set which is closed with respect to the constructible topology. A subset $Y$ of $X$ is said to be \emph{retrocompact in} $X$ provided that, for every open and quasi-compact subset $\Omega$ of $X$, $Y\cap \Omega$ is quasi-compact. The following well-known fact (whose proof can be found e.g. in \cite[Pag. 45]{ho} or in \cite[Theorem 2.1.3]{di-sc-tr}) provides a relation between the notions given above and will be freely used in what follows.
\begin{Prop}\label{pro-retro-spec}
Let $X$ be a spectral space and let $Y\subseteq X$. The following conditions are equivalent. 
\begin{enumerate}[\rm (i)]
\item $Y$ is proconstructible.
\item $Y$ is retrocompact and spectral (with the subspace topology of $X$).
\end{enumerate}
Furthermore, if $Y$ is proconstructible (and thus, in particular, spectral), then the subspace topology on $Y$ induced by the constructible topology of $X$ is the constructible topology of $Y$. 
\end{Prop}
%\begin{proof}
%If $Y$ is proconstructible, the fact that $Y$ is retrocompact and spectral follows from the fact that the constructible topology is Hausdorff. Conversely, if $Y$ is retrocompact and spectral, then the inclusion map $Y\hookrightarrow X$ is a spectral map, and so $Y$ is proconstructible by either \cite[p. 45]{ho} or \cite[1.9.5(vii)]{EGA4-1}.
%\end{proof}

\subsection{The order induced by a spectral topology.} Let$(P,\preceq)$ be a partially ordered set $(P,\preceq)$. For every subset $Q$ of $P$, set 
$$
Q^\uparrow:=\{p\in P\mid q\preceq p, \mbox{ for some }q\in Q\} 
$$
and
$$
Q^\downarrow:=\{p\in P\mid p\preceq q,\mbox{ for some }q\in Q\}.
$$
We say that $Q$ is \emph{closed under specializations} (resp., \emph{closed under generizations}) if $Q=Q^\uparrow$ (resp., $Q=Q^\downarrow$). 

Now, let $X$ be any topological space. A natural preorder can be defined on $X$ by setting, for every $x,y\in X$, $x\leq y:\iff y\in \chius(\{x\})$. In particular, if $\Omega$ is an open neighborhood of $y$ and $x\leq y$, then $x\in \Omega$.  Since every spectral space is, in particular, a T$_0$ space, the canonical preorder induced by a spectral topology is in fact a partial order, called \emph{the specialization order}. \begin{comment}
For every subset $Y$ of $X$, define the \emph{closure under generization of $Y$} to be the set
$$
Y^\downarrow:=\{x\in X\mid x\leq y, \mbox{ for some } y\in Y \}. 
$$
Clearly $Y\subseteq Y^\downarrow$, and when $Y=Y^\downarrow$ then $Y$ is said to be \emph{closed under generizations}.
\end{comment} 
Thus, if $X$ is a spectral space and $\leq $ is its specalization order, we can consider the partially ordered set $(X,\leq)$. It is well-known that, if $Y\subseteq X$ is quasi-compact, then $Y^\downarrow$ is proconstructible in $X$ (see, for instance,  \cite[Proposition 2.6]{fifolo2}). 

If $X$ is a spectral space (more generally, if $X$ is quasi-compact and $T_0$) then for every $x\in  X$ there is a maximal point $y$ such that $x\leq y$; in particular, $X=(\Max(X))^\downarrow$, where $\Max(X)$ is the set of maximal points of $X$ (see \cite[Remark 2.2(vi)]{Sch-Tre} or \cite[Proposition 4.1.2]{di-sc-tr}). Using the inverse topology (see Section \ref{sect:inverse}) we see that, dually, for every $x\in X$ there is a minimal point $z$ such that $z\leq x$.

We will also need some topologies defined on partially ordered sets. For a deeper insight in this circle of ideas see \cite[Section 7.1 and Appendix A.8]{di-sc-tr}. 

Let $(X,\leq)$ be a partially ordered set. The \emph{coarse lower topology} on $X$ associated to  $\leq$, denoted by $\tau^\ell(X)$ or $\tau^\ell(\leq)$, is the topology for which the sets of the type $\{x \}^\uparrow$, for $x$ varying in $X$, form a subbasis for the closed sets. The order induced by $\tau^\ell(X)$ is again $\leq$, and $\tau^\ell(X)$ is the coarsest topology inducing $\leq$. In general, the coarse lower topology is not spectral.

Suppose that $(X,\leq)$ is a \emph{direct complete partial order} (in brief, \emph{dcpo}), that is, suppose that every up-directed subset has a supremum. The \emph{Scott topology} on $X$ associated to $\leq$, denoted by $\sigma(X)$, is the topology whose open sets are the subsets $U$ of $X$ satisfying the following conditions:
\begin{enumerate}
\item $U$ is an up-set for $\leq$;
\item whenever $D$ is an up-directed subset of $X$ and $\sup(D)\in U$, then $U\cap D\neq \emptyset$. 
\end{enumerate}
The order induced by the Scott topology is the \emph{opposite} of $\leq$. It is possible to characterize when the Scott topology is spectral \cite[Theorem 7.1.21]{di-sc-tr}.

\subsection{The inverse topology.}\label{sect:inverse}
Let $X$ be a spectral space. Following \cite[Proposition 8]{ho} and \cite[Section 1.4]{di-sc-tr}, the \emph{inverse topology} on $X$ is the topological space $X^\inverse$ on the same base set of $X$, whose closed sets are the intersections of the open and quasi-compact subspaces of $X$. The inverse topology is spectral, and the order it induces is exactly the reverse of the order induced by the original spectral topology. If $Y\subseteq X$, we denote by $\chius^\inverse(Y)$ the closure of $Y$ in $X^\inverse$. By definition, for every $x\in X$, we have
$$
\chius^\inverse(\{x\})=\{y\in X \mid y\leq x \}. 
$$
The constructible topology of $X^\inverse$ coincides with the constructible topology of the starting topology.

Let now
$$
\xcal(X):=\{H\subseteq X \mid H\neq \emptyset, H \mbox{ is closed in  } X^\inverse\}. $$
As in \cite{fi-fo-spi-vietoris}, we endow $\xcal(X)$ with the so-called \emph{upper Vietoris topology}, i.e., the topology for which a basis of open sets is given by the sets
$$
\mathcal U(\Omega):=\{H\in \xcal(X)\mid H\subseteq \Omega \},
$$
where $\Omega$ runs among the open and quasi-compact subspaces of $X$. In \cite[Theorem 3.4]{fi-fo-spi-vietoris} it is proven that $\xcal(X)$ is a spectral space and that the canonical map $X\to \xcal(X)$, $x\mapsto \chius^\inverse(\{x\})=\{x\}^\downarrow$ is a spectral map and a topological embedding.

\section{Suprema of subsets and the constructible closure}\label{sect:sup}
If $X$ is a spectral space and $Y$ is a nonempty subset of $X$, we shall denote by $\sup(Y)$ the supremum of $Y$ (if it exists) in $X$, with respect to the order induced by the spectral topology. Furthermore, we define
$$
Y_f:=\{x\in X\mid x=\sup (F), \mbox{ for some }\emptyset \neq F\subseteq Y,F \mbox{ finite} \}, 
$$
$$
Y_\infty:=\{x\in X\mid  x=\sup(Z), \mbox{ for some } \emptyset \neq Z\subseteq Y\}.
$$

We say that \emph{$Y_f$ exists} if $\sup(F)$ exists for every nonempty finite subset $F\subseteq Y$, while we say that \emph{$Y_\infty$ exists} if $\sup(Z)$ exists for every nonempty subset $Z\subseteq Y$.

In the same way, we define $Y_{(f)}$ as the set of (existing) infima of the finite subsets of $Y$, and $Y_{(\infty)}$ as the set of (existing) infima of arbitrary subsets of $Y$. Likewise, we use  \emph{$Y_{(f)}$ exists} and \emph{$Y_{(\infty)}$ exists}, respectively, if $\sup(Z)$ exists for every finite $Z\subseteq Y$ (resp., for every $Z\subseteq Y$). Note that the infimum of a $Z\subseteq X$ is exactly the supremum of $Z$ with respect to the inverse topology, and thus results about $Y_{(f)}$ and $Y_{(\infty)}$ are often ``dual'' with the ones about $Y_f$ and $Y_\infty$.

In this paper, we are mainly interested in studying the relationship between existence of suprema and infima and the constructible topology. The following criterion will be extensively used through the paper.
\begin{theorem}\label{supremum}
Let $X$ be a spectral space and let $Y\subseteq X$.
\begin{enumerate}[\rm (1)]
\item\label{supremum:sup} If $Y_f$ exists, then $Y_\infty$ exists and $Y_\infty\subseteq\chius^\cons(Y_f)$.
\item\label{supremum:inf}  If $Y_{(f)}$ exists, then $Y_{(\infty)}$ exists and $Y_{(\infty)}\subseteq\chius^\cons(Y_{(f)})$.
\end{enumerate}
\end{theorem}
\begin{proof}
\ref{supremum:sup} Consider a nonempty subset $Z$ of $Y$. By assumption, $Z_f$ exists and clearly $Z_f\subseteq Y_f$. Furthermore, $Z_f$ is up-directed, with respect to the order induced by the spectral topology of $X$. Now apply \cite[Theorem 4.2.6]{di-sc-tr} to $Z_f$ to infer that $\sup(Z_f)$ exists and $\sup(Z_f)\in \chius^{\rm cons}(Z_f)\subseteq \chius^{\rm cons}(Y_f)$. The conclusion follows by noting that $\sup(Z)=\sup(Z_f)$.

\ref{supremum:inf} is the same result, but for the inverse topology (recall that $(X^\inverse)^\cons=X^\cons$).
\end{proof}

\begin{corollary}\label{infinity}
Let $X$ be a spectral space and let $Y\subseteq X$. If $Y_f$ exists, then $Y_f$ and $Y_\infty$ have the same closure, with respect to the constructible topology. 
\end{corollary} 
\begin{proof}
By Theorem \ref{supremum}, $Y_\infty$ exists and $Y_\infty\subseteq\chius^\cons(Y_f)$. Since $Y_f\subseteq Y_\infty$, we must have $\chius^\cons(Y_f)\subseteq\chius^\cons(Y_\infty)\subseteq\chius^\cons(Y_f)$.
\end{proof}

In general, we cannot recover the constructible topology exclusively from suprema and infima, because of two separate problems. The first one is that $Y_f$ need not to exist: for example, if the topology of $X$ is already Hausdorff (i.e., it coincides with the constructible topology), then no set with two or more elements has maximum. The second issue is that, even if $Y_f$ and $Y_{(f)}$ exist, the constructible closure of $Y$ may not be limited to $Y_\infty\cup Y_{(\infty)}$, since it is possible for limit points to be incomparable with elements of $Y$ (see Example \ref{ex:almded}).

Nevertheless, there are some cases where suprema and infima do determine the constructible closure.

\begin{Prop}\label{prop:chiuscons-linord}
Let $X$ be a spectral space whose specialization order is total, and let $Y\subseteq X$. Then, the following hold.
\begin{enumerate}[\rm (1)]
\item\label{prop:chiuscons-linord:cup} $\chius^\cons(Y)=Y_\infty\cup Y_{(\infty)}$.
\item\label{prop:chiuscons-linord:closed} $Y$ is closed in the constructible topology if and only if the supremum and the infimum of every nonempty subset of $Y$ belong to $Y$.
\end{enumerate}
\end{Prop}
\begin{proof}
\ref{prop:chiuscons-linord:cup} Since the specialization order of $X$ is total, we have $Z=Z_f$ for every $Z\subseteq Y$; hence, $\sup(Z),\inf(Z)$ exist and are in $Y$ by Theorem \ref{supremum}, i.e., $Y':=Y_\infty\cup Y_{(\infty)}\subseteq\chius^\cons(Y)$.

Now take any element $z\notin Y'$. Let $H:=\{y\in Y\mid y< z\}$, and let
\begin{equation*}
\widehat{H}:=\begin{cases}
\{x\in X\mid x>\sup(H)\} & \text{if~}H\neq\emptyset\\
X & \text{otherwise}.
\end{cases}
\end{equation*}
In the same way, set $K:=\{y\in Y\mid y>z\}$ and 
\begin{equation*}
\widehat{K}:=\begin{cases}
\{x\in X\mid x<\inf(K)\} & \text{if~}K\neq\emptyset\\
X & \text{otherwise}.
\end{cases}
\end{equation*}
By \cite[Theorem 1.6.4(ii)(b)]{di-sc-tr}, the constructible topology on $X$ is the interval topology with respect to the specialization order, i.e., it is the topology having the family of sets $\{x\}^\uparrow$ and $\{y\}^\downarrow$ as a subbasis of closed sets (see e.g. \cite[A.8(iv)]{di-sc-tr}). Hence, $\widehat{H}$ and $\widehat{K}$ are both open in the constructible topology, and thus $Z:=\widehat{H}\cap\widehat{K}$ is open. Since $\sup(H)$ and $\inf(K)$ (when they exist) are in $Y$, we have $z\in Z$; on the other hand, $Z\cap Y=\emptyset$ by construction. It follows that $z\notin\chius^\cons(Y)$, and thus $Y=Y'$, as claimed. 

\ref{prop:chiuscons-linord:closed} follows from \ref{prop:chiuscons-linord:cup}.
\end{proof}

We now introduce another class of spectral space; while the property defining them is restrictive, it actually holds for several spectral spaces of algebraic interest, as for example the set of submodules of a module or the set of overrings of an integral domain (see Sections \ref{sect:submodules} and \ref{sect:overrings}).
\begin{definition}\label{defin:locwithmax}
Let $X$ be a spectral space. We say that \emph{$X$ is locally with maximum} if every point of $X$ admits a local basis of open sets each of which has maximum, with respect to the order induced by the topology. 
\end{definition}

\begin{remark}\label{loc-with-maximum}
Let $X$ be a spectral space. If $U$ is a subset of $X$ with maximum $u_0$, then $U$ is quasi-compact (since, if $\mathcal A$ is an open cover of $U$ and $A\in\mathcal A$ contains $u_0$, then $A\supseteq U$). If furthermore $U$ is open, since $\chius^\inverse(\{u_0\})=\{x\in X\mid x\leq u_0 \}$, we immediately infer $U=\chius^\inverse(\{u_0\})$.
\end{remark}

We now relate the class of spectral spaces that are locally with maximum with other classes of topologies naturally arising starting from partially ordered sets. 
\begin{Prop}\label{prop:locwithmax}
Let $(X,\mathcal T)$ be a spectral space with specialization order $\leq$. Then, each of the following property implies the next ones.
\begin{enumerate}[\rm (i)]
\item\label{prop:locwithmax:locwithmax} $X$ is locally with maximum;
\item\label{prop:locwithmax:scott} $\mathcal T=\sigma(X^\inverse)$;
\item\label{prop:locwithmax:max} every open and quasi-compact subset of $X$ has finitely many maximal elements;
%\item\label{prop:locwithmax:max} if $\Omega$ is an open and compact subset of $X$, then $\Omega=(\Omega^{\mathrm{max}})^{\mathrm{gen}}$ and $\Omega^{\mathrm{max}}$ is finite;
\item\label{prop:locwithmax:tau} $\mathcal T^\inverse=\tau^U(X)=\tau^\ell(X^\inverse)$ (where $\mathcal T^{\rm inv}$ is the inverse topology of $\mathcal T$).
\end{enumerate}
Furthermore, if $(X,\leq)$ is a complete lattice then \ref{prop:locwithmax:scott}, \ref{prop:locwithmax:max} and \ref{prop:locwithmax:tau} are equivalent. 
\end{Prop}
\begin{proof}
\ref{prop:locwithmax:locwithmax} $\Longrightarrow$ \ref{prop:locwithmax:scott} Let $\Omega$ be an open subset of $X$ and let $D\subseteq X$ be lower-directed and such that $\inf(D)\in \Omega$. Then, by \cite[Proposition 6]{ho}, $\inf(D)\in \overline{D}\cap \Omega$. Since $\Omega$ is an open neighborhood of $\inf(D)$, it follows $\Omega\cap D\neq \emptyset$. Since $\Omega$ is closed by generizations, it follows that $\Omega$ is Scott-open, and thus $\mathcal T$ is coarser than the Scott topology on the opposite order of $X$, i.e., than $\sigma(X^\inverse)$.

Conversely, suppose that $\Omega$ is open in $\sigma(X^\inverse)$, and let $x\in\Omega$. Since $X$ is locally with maximum, there is a subset $D$ of $X$ such that $\mathcal{B}:=\{\{d\}^ {\downarrow}\mid d\in D \}$ is a local open basis at $x$ (with respect to the given spectral topology $\mathcal T$). If $d,e\in D$, then $V:=\{d\}^{\downarrow}\cap \{e\}^{\downarrow}$ is an open neighborhood of $x$ and, since $\mathcal B$ is local basis at $x$, there is a point $f\in D$ such that $\{f\}^{\downarrow}\subseteq V$. This proves that $D$ is lower-directed. Moreover, it is easily seen that $x=\inf(D)$. Since $\Omega\in \sigma(X^\inverse)$, we can pick a point $d\in D\cap \Omega$ and, since $\Omega$ is closed under generizations, it follows $x\in \{d\}^{\downarrow}\subseteq \Omega$. This proves that $\Omega\in\mathcal T$, i.e., $\mathcal T=\sigma(X^\inverse)$.

\medskip

\ref{prop:locwithmax:scott} $\Longrightarrow$ \ref{prop:locwithmax:max} Let $\Omega$ be an open and quasi-compact subset of $X$. By hypothesis, $\Omega$ is closed in the Scott topology of $X^\inverse$, and thus by \cite[7.1.8(viii)]{di-sc-tr} it is the downset of a finite set. The claim follows.

\medskip

\ref{prop:locwithmax:max} $\Longrightarrow$ \ref{prop:locwithmax:tau} The family $\mathring{\mathcal{K}}(X)$ of open and quasi-compact subsets of $X$ is a subbasis of closed sets of the inverse topology. Since $X$ is a spectral space, if $\Omega\in\mathring{\mathcal{K}}(X)$ then $\Omega=(\Omega^{\mathrm{max}})^\downarrow$; by hypothesis, $\Omega^{\mathrm{max}}$ is finite. Hence, the sets $\{p\}^\downarrow$, as $p$ ranges in $X$, form a subbasis of closed sets of $X^\inverse$. Therefore, $X^\inverse$ coincide with the coarse upper topology, as claimed.

\medskip

Suppose now that $(X,\leq)$ is a complete lattice and that \ref{prop:locwithmax:tau} holds. Then, in particular, $\tau^\ell(X^\inverse)$ is a spectral space, and thus by \cite[Theorem 7.2.8]{di-sc-tr} $\sigma(X^\inverse)=(\tau^\ell(X^\inverse))^\inverse$ is spectral. However, using the hypothesis,
\begin{equation*}
\mathcal T=(X^\inverse)^\inverse=(\tau^\ell(X^\inverse))^\inverse=\sigma(X^\inverse)
\end{equation*}
and thus \ref{prop:locwithmax:scott} holds, as claimed.
\end{proof}
\begin{remark}
\ref{prop:locwithmax:max} does not imply \ref{prop:locwithmax:locwithmax}. Let $X$ be the topological space constructed in the following way:
\begin{itemize}
\item $X=\{0,x,y_1,\ldots,y_n,\ldots\}$;
\item the nonempty closed sets are $X$, $\{x\}$, $\{y_1,\ldots,y_n\}$ and $\{x,y_1,\ldots,y_n\}$ for every $n\in \mathbb N$. 
\end{itemize}
Then, the order on $X$ is the following: $0$ is the minimal point, $x, y_1$ are maximal, there's no element between $0$ and $x$ and $y_1>y_2>\cdots$ is a descending chain without minimal elements. It is clear that $X$ is spectral: indeed, any open set of $X$ is quasi-compact, and the irreducible closed sets are precisely $X=\overline{\{0\}},\{x\}$ and $C_n:=\{y_1,\ldots, y_n\}=\overline{\{y_n\}}$, for all $n\geq 1$. 

Every open set (indeed, every set) has at most two maximal elements, and thus $X$ satisfy \ref{prop:locwithmax:max}. On the other hand, if $\Omega$ is an open set containing $x$, then  $\Omega=\{x,y_k,y_{k+1},\ldots\}$ for some $k\in \mathbb N$, and thus no open set containing $x$ has a maximum. It follows that $X$ is not locally with maximum.

The example actually shows that \ref{prop:locwithmax:scott} does not imply \ref{prop:locwithmax:locwithmax}, since every set closed by generizations is open with respect to $\sigma(X^\inverse)$, and thus $\sigma(X^\inverse)$ is equal to the given spectral topology.
\end{remark}

To use this definition, we need the following connection between closures of different topologies.
\begin{lemma}\label{closure-sup}
Let $X$ be a spectral space, and let $\mathcal T$ be a topology on $X$ such that any open and quasi-compact subset of $X$ is closed, with respect to $\mathcal T$. If $Y\subseteq Z$ are  nonempty subsets of $X$, $\chius^\mathcal{T}(Y)$ is the closure  of $Y$ with respect to $\mathcal T$ and there exists $\sup(Y)\in X$, then there exists $\sup(\chius^\mathcal{T}(Y)\cap Z)$, and we have $\sup(Y)=\sup(\chius^\mathcal{T}(Y)\cap Z)$. 
\end{lemma}
\begin{proof}
Since $Y\subseteq\chius^\mathcal{T}(Y)\cap Z$, it is sufficient to prove that any upper bound $x \in X$ for $Y$ is an upper bound for $\chius^\mathcal{T}(Y)\cap Z$ too. Assume that there exists an element $z\in\chius^\mathcal{T}(Y)\cap Z$ such that $z\nleq x$ (i.e., $x\notin\chius(\{z\})$). Since $X$ is a spectral space, there exists an open and quasi-compact subset $\Omega$ of $X$ such that $x\in \Omega$ and $z\notin \Omega$. Since $\Omega$ is closed with respect to $\mathcal T$, it follows that $Y$ can't be contained in $\Omega$ (otherwise $\chius^\mathcal{T}(Y)\subseteq\Omega$); hence, there is an $y\in Y\cap (X\setminus\Omega)$. Since $\Omega$ is open in the starting topology, $\chius(\{y\})\subseteq X\setminus\Omega$, and so $x\notin\chius(\{y\})$, that is, $y\nleq x$, against the fact that $x$ is an upper bound of $Y$ in $X$. 
\end{proof}

\begin{Prop}\label{density-locally-with-maximum}
Let $X$ be a spectral space which is locally with maximum and let $Y$ be a proconstructible subset of $X$ such that $Y_f$ exists. Then:
\begin{enumerate}[\rm (1)]
\item\label{density-locally-with-maximum:f} $Y_\infty$ exists and $\chius^\cons(Y_f)=Y_\infty$;
\item\label{density-locally-with-maximum:fixn} for every integer $n\geq 1$, the set $Y_n:=\{\sup(F)\mid F\subseteq Y, |F|=n\}$ is proconstructible in $X$;
\item\label{density-locally-with-maximum:dense} if $Z\subseteq Y$ is dense in $Y$, then $Z_f$ is dense in $Y_\infty$.
\end{enumerate}
\end{Prop}
\begin{proof}
\ref{density-locally-with-maximum:f} Since $Y$ is proconstructible in $X$, it is spectral, and thus it makes sense to consider the hyperspace $\xcal(Y)$, endowed with the upper Vietoris topology. By Theorem \ref{supremum}, the map $\Sigma:\xcal(Y)\longrightarrow X$ defined by setting $\Sigma(H):=\sup(H)$, for each $H\in\xcal(Y)$, is well-defined. Since the space is locally with maximum, by \cite[Lemma 4.6]{fi-fo-spi-vietoris} $\Sigma$ is a spectral map; hence, it is continuous when $\xcal(Y)$ and $X$ are equipped with their constructible topologies, and thus $\Sigma$ is a closed map (again in the constructible topology). It follows, in particular, that $\Sigma(\xcal(Y))$ is proconstructible. On the other hand, by Lemma \ref{closure-sup} (applied by taking as $\mathcal T$ the inverse topology on $X$) we infer that $\Sigma(\xcal(Y))=Y_\infty$. Now the conclusion follows Corollary \ref{infinity}. 

\ref{density-locally-with-maximum:fixn} Consider the map
\begin{equation*}
\begin{aligned}
\Sigma_n\colon Y^n& \longrightarrow \xcal(X)\\
(y_1,\ldots,y_n) & \longmapsto \{y_1,\ldots,y_n\}^\downarrow,
\end{aligned}
\end{equation*}
where $Y^n$ is endowed with the product topology of the topology induced by $X$. Since $Y$ is a spectral space (being proconstructible) then so is $Y^n$, by \cite[Theorem 7]{ho}. 

If $\Omega$ is open and quasi-compact in $X$, then
\begin{align*}
\Sigma_n^{-1}(\Omega)= & \{(y_1,\ldots,y_n)\mid \{y_1,\ldots,y_n\}^\downarrow\subseteq\Omega\}=\\
=& \{(y_1,\ldots,y_n)\mid y_1,\ldots,y_n\in\Omega\}=\Omega^n,
\end{align*}
which is quasi-compact as it is a product of quasi-compact spaces. Hence, $\Sigma_n$ is a spectral map, and in particular $\Sigma_n(Y^n)$ is closed in the constructible topology. By the previous part of the proof, it follows that $\Sigma(\Sigma_n(Y^n))=Y_n$ is closed, as well.

\ref{density-locally-with-maximum:dense} Note first that $Z_f$ exists since $Z\subseteq Y$ and $Y_f$ exists by hypothesis. If $Z$ is dense in $Y$, then $Z^n$ is dense in $Y^n$; setting $\Psi_n:=\Sigma_n\circ\Sigma:Y^n\longrightarrow Y_\infty$ (with $\Sigma$ and $\Sigma_n$ as in the previous point), we see that $\Psi_n(Z^n)=Z_n$ is dense in $\Psi_n(Y^n)=Y_n$. Therefore, $Z_f=\bigcup_n Z_n$ is dense in $Y_f=\bigcup_nY_n$; by part \ref{density-locally-with-maximum:f}, it follows that $\chius^\cons(Z_f)=\chius^\cons(Y_f)=Y_\infty$, that is, $Z_f$ is dense in $Y_\infty$. The claim is proved.
\end{proof}

\begin{remark}
The fact that the starting set $Y$ is proconstructible is critical in the hypothesis of Proposition \ref{density-locally-with-maximum}. For example, if $X$ is a spectral space whose specialization order is total, then $Y=Y_f$ for all subspaces $Y$, but $\chius^\cons(Y)$ may contain elements beyond $Y_\infty$ (for example, the infimum of its subsets).
%For example, suppose $X:=\{0,y_1,y_2,\ldots\}$ is endowed with the topology having as closed sets $X$ itself and the spaces $\{y_1,\ldots,y_n\}$, for each $n\inN$. Then, $X$ is a spectral space whose specialization order is total; indeed, $0$ is the minimum while the $y_i$ are a descending chain $y_1>y_2>\cdots$. If $Y:=\{y_1,y_2,\ldots\}$, then $Y=Y_f=Y_\infty$, but $\chius^\cons(Y)=X$.
\end{remark}

We end this section by giving an example where the constructible closure is larger than $Y_\infty\cup Y_{(\infty)}$. Recall that an \emph{almost Dedekind domain} is an integral domain $D$ such that $D_{\f m}$ is a discrete valuation ring for every maximal ideal $\f m$ of $D$.
\begin{example}\label{ex:almded}
Let $D$ be a non-Noetherian almost Dedekind domain and let $\mathfrak{n}$ be a non-finitely generated maximal ideal of $D$ (for explicit examples of almost Dedekind domains which are not Dedekind see, for instance, \cite[p. 426]{Nakano} and \cite{lo95}). Let $X_0:=\Spec(D)$, and define a topological space $X$ in the following way: as a set, $X$ is the disjoint union of $X_0$ and an element $\infty\notin X_0$, while the open sets of $X$ are $X$ itself and the open sets of $X_0$. On $X_0$, the order induced by this topology is the the same order of $X_0$, while $\infty$ is bigger than every element of $X$ and so is the unique maximal element of $X$. In particular, $X$ is T$_0$ and quasi-compact. 

The open and quasi-compact subsets of $X$ are the open and quasi-compact subsets of $X_0$, plus $X$ itself, and thus they form a basis closed by finite intersection. Furthermore, the nonempty irreducible closed subsets of $X$ are $\{\infty\}$ and the sets $C\cup\{\infty\}$, where $C$ is a nonempty irreducible closed subset of $X_0$, and thus in particular they have a generic point. Thus, $X$ is spectral.

%We claim that $X$ is spectral. If $\mathcal{S}_0$ is the basis of $X_0$ formed by all open and quasi-compact subsets, then $\mathcal{S}:=\mathcal{S}_0\cup\{X\}$ is a basis of $X$. Let $\mathscr{U}$ be an ultrafilter on $X$. If $X_0\notin\mathscr{U}$, then $\{\infty \}\in \ms U$ (i.e., $\mathscr{U}$ is the principal ultrafilter based on $\infty$) and thus $\infty$ is an ultrafilter limit point of $\mathscr{U}$. If $X_0\in\mathscr{U}$, then $\ms U_0:=\{U\cap X_0\mid U\in\ms U \}$ is an ultrafilter on $X_0$, and since $X_0$ is spectral it has an ultrafilter limit point $x$ with respect to $\ms U_0$, which is also an ultrafilter limit point of $X$ with respect to $\mathscr{U}$. By \cite[Corollary 3.3]{Fi}, $X$ is a spectral space.

Consider now $Y:=X\setminus\{\f n \}$. Then, every subset $H$ of $Y$ has supremum and infimum in $X$, and they belong to $Y$: as a matter of fact, if $|H\cap\Max(D)|\leq 1$ then $H$ is linearly ordered (and finite) while if $|H\cap\Max(D)|\geq 2$ then $(0)$ is the infimum of $H$ and $\infty$ is its supremum. We claim that $\mathfrak{n}\in\chius^\cons(Y)$; since $X_0$ is open and quasi-compact in $X$ (and so proconstructible) by Proposition \ref{pro-retro-spec} its constructible topology is the subspace topology of the constructible topology of $X$, and thus we need only to show that $\mathfrak{n}$ is in the constructible closure of $Y\cap X_0$ in $X_0$. If not, then $Y\cap X_0$ would be proconstructible, and thus in particular quasi-compact; hence, $Y\cap X_0$ is the open set induced by a finitely generated ideal $I$, and so $\mathfrak{n}$ is the radical of a finitely generated ideal. Since $D$ is almost Dedekind, by \cite[Theorem 36.4]{gilmer} it would follow that $I=\mathfrak{n}^k$ for some positive integer $k$; since an almost Dedekind domain is Pr\"ufer, $I$ is invertible, and so $\mathfrak{n}$ would be invertible, a contradiction. Hence, $\mathfrak{n}\in\chius^\cons(Y\cap X_0)\subseteq\chius^\cons(Y)$, as claimed.
\end{example}

\section{Algebraic lattices of sets}\label{sect:alglatt}
The main examples of spectral spaces to which we want to apply the results in the previous sections are spaces of submodules and of overrings. Those settings, and analogous spaces constructed from algebraic substructures, are best understood in the framework of \emph{algebraic lattices of sets} (see \cite[7.1.12]{di-sc-tr}).

Let $S$ be a nonempty set, and let $\mathfrak{P}(S)$ be its power set, ordered by inclusion. A family $\mathcal{L}\subseteq\mathfrak{P}(S)$ is an \emph{algebraic lattice of sets} if it satisfies the following properties:
\begin{enumerate}
    \item[\rm (a1)] $\mathcal L$ is closed under arbitrary intersections (in particular, $S\in \mathcal L$).
    \item[\rm (a2)] $\mathcal L$ is closed under nonempty up-directed unions.
\end{enumerate}
Examples of such families include the set of submodules of a module, the set of subrings of a fixed ring, or more generally the set of substructures of an algebraic structure. Topologically, such families behave very well, since every such $\mathcal{L}$ is a spectral space under the coarse lower topology associated to $\subseteq$ \cite[Part (iii) of 7.2.12]{di-sc-tr}; more precisely, $\mathcal{L}$ is a proconstructible subset of $\mathfrak{P}(S)$ (when $\mathfrak{P}(S)$ is endowed with the coarse lower topology), since the inclusion $\mathcal{L}\longrightarrow\mathfrak{P}(S)$ is a spectral map \cite[end of 7.2.12]{di-sc-tr}.

A different point of view on algebraic lattices of sets can be given through the concept of closure operation. A \emph{closure operation} on $S$ (or, simply, a \emph{closure}) is a map $c:\mathfrak{P}(S)\longrightarrow\mathfrak{P}(S)$, $I\mapsto I^c$, such that the following properties are satisfied for all $I,J\in\mathfrak{P}(S)$:
\begin{itemize}
\item $I\subseteq I^c$;
\item if $I\subseteq J$, then $I^c\subseteq J^c$;
\item $(I^c)^c=I^c$. 
\end{itemize}
We denote by $\mathfrak{P}(S)^c$ the set of \emph{$c$-closed} subsets of $S$, i.e., the set of the $I$ such that $I=I^c$. For any closure operation $c$ on $S$, one has 
$$
I^c=\bigcap\{J^c\mid J\supseteq I\},
$$
for all $I\in \mathfrak P(S)$.

Given two closure operations $c,d$, we also say that $c\leq d$ if $I^c\subseteq I^d$ for every $I\subseteq S$; equivalently, $c\leq d$ if and only if $\mathfrak P(S)^c\supseteq \mathfrak P(S)^d$. In particular, $c=d$ if and  only if $\mathfrak P(S)^c=\mathfrak P(S)^d$. 

For every closure operation $c$, the map $c_f:\mathfrak{P}(S)\longrightarrow\mathfrak{P}(S)$,
\begin{equation*}
c_f:J\mapsto\bigcup\{I^c\mid I\subseteq J,~|I|<\infty\}
\end{equation*}
is again a closure operation, and $(c_f)_f=c_f$. We say that $c$ is \emph{of finite type} if $c=c_f$. In particular, $c_f\leq c$.

\begin{Prop}\label{finite-type}
Let $S$ be a nonempty set and let $c$ be a closure operation on $S$.
\begin{enumerate}[\rm (1)]
\item $\mathfrak{P}(S)^c$ is closed by arbitrary intersections.
\item $\mathfrak{P}(S)^c$ is closed by up-directed unions if and only if $c$ is of finite type. 
\end{enumerate}
\end{Prop}
\begin{proof}
The first claim is obvious. If $c$ is not of finite type, let $I$ be such that $I^{c_f}\subsetneq I^c$; then, the family $\mathcal{F}:=\{F^c\mid F\subseteq I,~|F|<\infty\}$ is up-directed, but its union $I^{c_f}$ is not in $\mathfrak{P}(S)^c$.

Suppose now that $c$ is of finite type, let $\mathcal F$ be an up-directed subset of $\mathfrak{P}(S)^c$, and let $I$ be the union of the elements of $\mathcal{F}$. Take $x\in I^c$. Since $c$ is of finite type, there is some finite $J=\{j_1,\ldots,j_n\}\subseteq I$ such that $x\in J^c$. For every $i\in\{1,\ldots,n\}$, there is a set $I_i\in\mathcal{F}$ such that $j_i\in I_i^c$; hence, $x\in J^c\subseteq(I_1\cup\cdots\cup I_n)^c\subseteq I$, since $\mathcal F$ is up-directed. This proves that $I=I^c$, that is, $I\in \mathfrak{P}(S)^c$. The claim is proved.
\end{proof}
\begin{corollary}\label{alg-lat-cl-op}
Let $S$ be a set, and let $\mathcal L\subseteq \mathfrak P(S)$. The following conditions are equivalent.
\begin{enumerate}[\rm (i)]
    \item\label{alg-lat-cl-op:alg} $\mathcal L$ is an algebraic lattice of sets. 
    \item\label{alg-lat-cl-op:clos} There is a (unique) closure operation $c$ on $S$ of finite type such that $\mathcal L=\mathfrak P(S)^c$. 
\end{enumerate}
\end{corollary}
\begin{proof}
\ref{alg-lat-cl-op:clos}$\Longrightarrow$\ref{alg-lat-cl-op:alg} is Proposition \ref{finite-type}. Conversely, let $\mathcal L$ be satisfying \ref{alg-lat-cl-op:alg}. Define a map $c_\mathcal{L}$ by
\begin{equation*}
J^{c_\mathcal{L}}:=\bigcap\{I\in\mathcal{L}\mid J\subseteq I\},
\end{equation*}
for every $J\in\mathfrak{P}(S)$. Then, $c_\mathcal{L}$ is a closure operation on $S$, and since $\mathcal{L}$ is closed by arbitrary intersection we have $\mathcal{L}=\mathfrak{P}(S)^{c_\mathcal{L}}$. Since $\mathcal{L}$ is closed by up-directed unions, applying  Proposition \ref{finite-type} again we see that $c_\mathcal{L}$ is of finite type. The uniqueness of $c_{\mathcal L}$ follows from the discussion before Proposition \ref{finite-type}. 
\end{proof}
Therefore, talking about algebraic lattices of sets is equivalent to talking about finite-type closure operations.

Let now $\mathcal{L}$ be the algebraic lattice associated to the closure operation $c$, i.e., $\mathcal{L}=\mathfrak{P}(S)^c$. We say that $A\in\mathcal{L}$ is \emph{finitely generated} if $A=F^c$ for some finite set $F$. (This is equivalent to the definition in \cite[7.2.12]{di-sc-tr}, by the proof of Corollary \ref{alg-lat-cl-op}.) We denote the set of finitely generated elements of $\mathcal{L}$ as $\mathcal{L}_\fin$.

We call the inverse topology of the coarse lower topology on $\mathcal{L}$ the \emph{Zariski topology}, and we denote it by $\mathcal{L}^\zariski$. Using \cite[7.2.8(ii,d) and 7.2.12(i)]{di-sc-tr}, we see that a basis of open and quasi-compact subspaces for the Zariski topology is  formed by the sets of the type $\{F_1,\ldots, F_n\}^\uparrow$, where $F_1,\ldots, F_n\in \mathcal L$ are finitely generated. In case of spaces of overrings this topology coincides with the Zariski topology generalizing the Zariski-Riemann spaces of valuation rings (hence the name).
\begin{Prop}\label{prop:propL}
Let $\mathcal{L}$ be an algebraic lattice of sets.
\begin{enumerate}[\rm (1)]
\item\label{prop:propL:locwithmax} $\mathcal L^\zariski$ is locally with maximum.
\item\label{prop:propL:dense} $\mathcal L_{\rm fin}$ is dense in $\mathcal L^{\rm cons}$.
\item\label{prop:propL:noeth} $\mathcal{L}_\fin$ is spectral in the Zariski topology if and only if $\mathcal{L}_\fin=\mathcal{L}$.
\end{enumerate}
\end{Prop}
\begin{proof}
Let $c:\mathfrak P(S)\to \mathfrak P(S)$ be the closure operation of finite type inducing $\mathcal L$ (Corollary \ref{alg-lat-cl-op}). 

\ref{prop:propL:locwithmax} follows from the fact that the set $\{F\}^\uparrow$ has minimum $F$, for any $F\in\mathcal L$.

\ref{prop:propL:dense} $\mathcal L_{\rm fin}$ is closed under finite suprema: indeed, if if $I=F^c$ and $J=G^c$ are finitely generated subsets of $S$ (with $F,G$ finite), then $(I\cup J)^c=(F\cup G)^c$ is again finitely generated. Since $\mathcal{L}$ is closed by up-directed unions, every element of $\mathcal{L}$ is the union of the finitely generated sets which it contains, and thus $\mathcal L=(\mathcal L_{\rm fin})_\infty$. The conclusion follows from Theorem \ref{supremum}. 

\ref{prop:propL:noeth} We claim that $\mathcal{L}_\fin$ is retrocompact in $\mathcal L$, with respect to the Zariski topology. Indeed, take a subbasic open and quasi-compact subset $U:=\{F_1,\ldots, F_n\}^\uparrow$ of $\mathcal L^{\rm zar}$, and let $\mathcal A$ be an open cover of $U\cap \mathcal L_{\rm fin}$. We can assume that $F_1,\ldots,F_n\in\mathcal L_{\rm fin}$, and thus there are open sets $\Omega_i\in \mathcal A$ ($1\leq i \leq n)$) such that $F_i\in \Omega_i$. Then clearly $\{\Omega_1,\ldots,\Omega_n\}$ is a finite subcover for $U\cap \mathcal L_{\rm fin}$.  By Proposition \ref{pro-retro-spec}, it follows that $\mathcal{L}_\fin$ is spectral if and only if it is proconstructible; the claim now follows by part \ref{prop:propL:dense}.
\end{proof}

\begin{remark}\label{oss:closL}
Since $\mathcal{L}$ is itself a partially ordered set, it is possible to define closure operations on $\mathcal{L}$, and the concept of finitely generated elements allows to define the map $c\mapsto c_f$ (and thus closure operations of finite type) by substituting finite sets with finitely generated elements. While this is useful in some contexts (see the case of semistar operations in Section \ref{sect:semistar}, which are classically defined only on submodules), it is not needed from the theoretical point of view.

Indeed, let $c$ be the closure operation of finite type associated to $\mathcal{L}$. If $d:\mathcal{L}\longrightarrow\mathcal{L}$ is a closure operation, then $d\circ c$ is a closure operation on $S$ whose restriction to $\mathcal{L}$ coincides with $d$; conversely, if $e$ is a closure operation on $S$ that restricts to a closure operation on $\mathcal{L}$, then $c\leq e$ and $e=e|_{\mathcal{L}}\circ c$. Moreover, it is not hard to see that $d\circ c$ is of finite type if and only if $d$ is; therefore, talking about closure operations on $\mathcal{L}$ is equivalent to talking about closure operations that are bigger than $c$.
\end{remark}

There are several natural closure operations that are not of finite type: for example, if $S$ is endowed with a topology $\mathcal{T}$, then the closure operator $c$ with respect to $\mathcal{T}$ is a closure operation but, if all points are closed, then $c$ is of finite type if and only if the topology is discrete.% The easiest algebraic example is the \emph{$v$-operation} (also called \emph{divisorial closure}) on a domain $D$, which is defined by the double dual $v:I\mapsto(D:_K(D:_KI))$ (where $K$ is the quotient field of $D$); this is one of the major examples of a star operation, but it is usually not of finite type.

The following proposition shows how to ``control'' the set of $c$-closed sets where $c$ is not of finite type.
\begin{comment}
\begin{Prop}\label{prop:Fstar}
Let $c$ be a closure operation on $S$. Then,
\begin{equation*}
\chius^\cons(\{F^c\mid F\subseteq S, |F|<\infty\})=\chius^\cons(\mathfrak{P}(S)^c)=\mathfrak{P}(S)^{c_f}.
\end{equation*}
\end{Prop}
\begin{proof}
Since $c_f$ is of finite type, $\mathcal{L}:=\mathfrak{P}(S)^{c_f}$ is an algebraic lattice of set and thus it is closed in the constructible topology. The set $\{F^c: |F|<\infty\}$ is exactly $\mathcal{L}_\fin$, and thus it is dense in $\mathcal{L}$ by Proposition \ref{prop:propL}\ref{prop:propL:dense}. The equality $\chius^\cons(\mathcal{L}^c)=\mathcal{L}$ now follows from this and the fact that $\mathcal{L}_\fin\subseteq\mathcal{L}^c\subseteq\mathcal{L}$.
\end{proof}
\end{comment}
\begin{Prop}\label{prop:Fstar}
Let $\mathcal L$ be an algebraic lattice of sets and let $c$ be a closure operation on $\mathcal{L}$. Then,
\begin{equation*}
\chius^\cons(\{F^c\mid F\in\mathcal{L}_{\mathrm{fin}}\})=\chius^\cons(\mathcal{L}^c)=\mathcal{L}^{c_f}.
\end{equation*}
\end{Prop}
\begin{proof}
Since $c_f$ is of finite type,  $\mathcal{L}^{c_f}$ is an algebraic lattice of sets and  its set $(\mathcal L^{c_f})_{\rm fin}$ of finitely generated elements is $\{F^c\mid F\in\mathcal{L}_{\mathrm{fin}}\}$. By Proposition \ref{prop:propL}\ref{prop:propL:dense}, $(\mathcal L^{c_f})_{\rm fin}$ is dense in $\mathcal{L}^{c_f}$. Moreover, $\mathcal L^{c_f}$ is proconstructible in $\mathcal L$ (since both of them are proconstructible in $\mathfrak P(S)$, as remarked at the beginning of this section). 
Thus $\chius^\cons((\mathcal L^{c_f})_{\rm fin})=\mathcal{L}^{c_f}$.

The equality $\chius^\cons(\mathcal{L}^c)=\mathcal{L}^{c_f}$ now follows from this and the fact that $(\mathcal L^{c_f})_{\rm fin}\subseteq\mathcal{L}^c\subseteq\mathcal{L}^{c_f}$.
\end{proof}

The following result will be useful in the applications.
\begin{Prop}\label{isolated-fg}
Let $S$ be a set and let $\mathcal L\subseteq \mathfrak P(S)$ be an algebraic lattice of sets. Let $M\in \mathcal L$ be maximal in $\mathcal{L}\setminus\{S\}$. Then $M$ is isolated in $\mathcal L^{\rm cons}$ if and only if $M$ is finitely generated. 
\end{Prop}
\begin{proof}
If $M$ is finitely generated, then $\{M\}^\uparrow=\{M,S\}$ is open in $\mathcal L^{\rm zar}$ and, a fortiori, in $\mathcal L^{\rm cons}$. Moreover, $\{S\}$ is proconstructible (the constructible topology is Hausdorff). Thus $\{M\}=\{M\}^\uparrow\cap (\mathcal L\setminus\{S\})$ is open. 

Conversely, if $M$ is not finitely generated it cannot be isolated in $\mathcal L^{\rm cons}$, in view of Proposition \ref{prop:propL}\ref{prop:propL:dense}. 
\end{proof}

Let $\insclos_f(S):=\insclos_f(\mathfrak P(S))$ be the set of closure operations of finite type on $\mathfrak{P}(S)$. For every $I\subseteq S$ and every $x\in S$, let
\begin{equation*}
V(I,x):=\{c\in\insclos_f(S)\mid x\in I^c\}.
\end{equation*}
We call the topology generated by the $V(I,x)$ on $\insclos_f(S)$, where $I$ ranges among the subsets of $S$ and $x$ ranges in $S$, the \emph{Zariski topology} of $\insclos_f(S)$. The next proof essentially follows the proof of \cite[Theorem 2.13]{FiSp}.
\begin{Prop}\label{prop:closf-spectral}
The space $\insclos_f(S)$, endowed with the Zariski topology, is a spectral space.
\end{Prop}
\begin{proof}
Let $\mathcal{S}$ be the set of the $V(I,x)$, as $I$ ranges among the \emph{finite} subsets of $S$. Then, $\mathcal{S}$ is a subbasis of the Zariski topology on $\insclos_f(S)$. Every such $V(I,x)$ has a minimum, namely the map $c(I,x)$ defined by
\begin{equation*}
J^{c(I,x)}:=\begin{cases}
I\cup\{x\} & \text{if~}J=I,\\
J & \text{otherwise}.
\end{cases}
\end{equation*}
(The fact that $c(I,x)$ is of finite type follows from the fact that $I$ is finite.) Moreover, if  $\Lambda\subseteq\insclos_f(S)$, then $\Lambda$ has a maximum, namely the closure defined by
\begin{equation*}
J\mapsto \bigcup\{J^{c_1\circ\cdots\circ c_n}\mid c_1,\ldots,c_n\in\Lambda\};
\end{equation*}
the fact that this map is a closure follows exactly as in \cite[p.1628]{anderson-examples}.

Let now $\mathscr{U}$ be an ultrafilter on $X:=\insclos_f(S)$. Let
\begin{equation*}
c:=\sup\{c(I,x)\mid V(I,x)\in\mathscr{U}\}.
\end{equation*}
We claim that $c\in V(I,x)$ if and only if $V(I,x)\in\mathscr{U}$. Indeed, if $V(I,x)\in\mathscr{U}$ then $c(I,x)\leq c$, so that $x\in I^{c(I,x)}\subseteq I^c$ and $c\in V(I,x)$.

Conversely, suppose $c\in V(I,x)$. Then, by definition and by the first paragraph of the proof, there are finite subsets $J_1,\ldots,J_k$ of $S$ and $x_1,\ldots,x_k\in S$ such that $V(J_i,x_i)\in\mathscr{U}$ for every $i$ and $x\in I^{c(J_1,x_1)\circ\cdots\circ c(J_k,x_k)}$. Let $U:=\bigcap_iV(J_i,x_i)$, and take $d\in U$: then, $d\geq c(J_i,x_i)$ for every $i$, and thus
\begin{equation*}
x\in I^{c(J_1,x_1)\circ\cdots\circ c(J_k,x_k)}\subseteq I^{d\circ\cdots\circ d}=I^d,
\end{equation*}
that is, $d\in V(I,x)$; hence, $U\subseteq V(I,x)$. Since $\mathscr{U}$ is an ultrafilter, $U\in\mathscr{U}$ and thus $V(I,x)\in\mathscr{U}$, as claimed.

Hence, $c$ is the ultrafilter limit point of $\mathscr{U}$ with respect to $\mathcal{S}$. Since the Zariski topology is clearly $T_0$, by \cite[Corollary 3.3]{Fi} it follows that $\insclos_f(S)$ is a spectral space.
\end{proof}
Let $\mathcal L\subseteq \mathfrak P(S)$ be an algebraic lattice of sets and  $\insclos_f(\mathcal L)$ be  the set of finite-type closure operations on $\mathcal L$. The Zariski topology can be defined on $\insclos_f(\mathcal L)$ by declaring as subbasic open sets the sets $V_{\mathcal L}(L,x):=\{r\in \insclos_f(\mathcal L)\mid x\in L^r\}$, where $L\in\mathcal L$, $x\in S$. 
\begin{corollary}\label{cor:closL}
Let $\mathcal{L}$ be an algebraic lattice of sets, and let $c\in\insclos_f(S)$ be the finite-type closure operation associated to $\mathcal{L}$. Then, $\insclos_f(\mathcal L)$ (endowed with the Zariski topology) is a spectral space, and the map
\begin{equation*}
\begin{aligned}
\lambda_\mathcal{L}\colon\insclos_f(\mathcal{L})\longrightarrow & \insclos_f(S)\\
r\longmapsto & r\circ c
\end{aligned}
\end{equation*}
is a spectral map and a topological embedding. In particular, $\insclos_f(\mathcal L)$ is proconstructible in $\insclos_f(S)$.
\end{corollary}
\begin{proof}
By Remark \ref{oss:closL}, the mapping $\lambda_\mathcal{L}$ is injective, with image $\{c\}^\uparrow$. It is straightforward to see that $\lambda_\mathcal{L}$ is a topological embedding; the claims now follow by noting that $\{c\}^\uparrow$ is closed in the inverse topology of $\insclos_f(S)$ (and thus, in particular, it is proconstructible).
\end{proof}

\begin{comment}
Let $c\in \insclos_f(S)$ be associated to $\mathcal L$ (Corollary \ref{alg-lat-cl-op}). By Remark \ref{oss:closL}, the mapping $\lambda: \insclos_f(\mathcal L)\to \insclos_f(S)$, $r\mapsto r\circ c$ is injective and it is clearly a topological embedding, with respect to the Zariski topology. Moreover, $\lambda(\insclos_f(\mathcal L))=\{c\}^\uparrow$ is closed in the inverse topology of the Zariski topology and, a fortiori, it is proconstructible in the spectral space $\insclos_f(S)$. Then the following corollary is now clear. 
\begin{corollary}\label{cor:closL}
Let $\mathcal{L}\subseteq\mathfrak{P}(S)$ be an algebraic lattice of sets. Then, the set $\insclos_f(\mathcal{L})$ is a spectral space, with respect to the Zariski topology. 
%of finite-type closure operations on $\mathcal{L}$ is a proconstructible subset of $\insclos_f(S)$, with respect to the Zariski topology.
\end{corollary}
%\begin{proof}
%Let $c$ be the finite-type closure operation on $S$ associated to $\mathcal{L}$. By Remark \ref{oss:closL}, $\insclos_f(\mathcal{L})=\{c\}^\uparrow$, which is closed in the inverse topology of the Zariski topology; in particular, it is proconstructible.
%\end{proof}
\end{comment}

\begin{remark}
Let $S$ be a set. Let $\mathscr{C}$ be the family of all subsets $\mathcal{L}\subseteq\mathfrak{P}(S)$ that are algebraic lattices of sets: by Proposition \ref{alg-lat-cl-op}, the elements of $\mathscr{C}$ correspond to the finite-type closure operations on $S$, and thus there is a natural bijective correspondence between $\mathscr{C}$ and $\insclos_f(S)$. A natural question is whether $\mathscr{C}$ \emph{itself} is an algebraic lattice of sets; in this case, Proposition \ref{prop:closf-spectral} would be a simple consequence of the general theory.

It is not hard to see that $\mathscr{C}$ is closed by arbitrary intersections; therefore, $\mathscr{C}$ defines a closure operation on $\mathfrak{P}(S)$, i.e., a map $\mathfrak{P}(\mathfrak{P}(S))\longrightarrow\mathfrak{P}(\mathfrak{P}(S))$ sending $\mathcal{L}$ to the smallest algebraic lattice of sets containing $\mathcal{L}$. In terms of closure operations, this is equivalent to saying that the supremum of a family of finite-type operations is still of finite type.

However, $\mathscr{C}$ is usually not closed by up-directed unions. Indeed, let $\mathscr{F}$ be an up-directed subset of $\mathscr{C}$. Then, the set $\mathfrak{A}:=\{c_\mathcal{L}\mid\mathcal{L}\in\mathscr{F}\}$ is a subset of $\insclos_f(S)$, which is down-directed in the natural order of $\insclos_f(S)$. The union of $\mathscr{F}$ corresponds to the infimum of $\mathfrak{A}$ \emph{in the set of all closures}, and this closure does not need to be of finite type. For example, let $\Delta$ be a non-quasi-compact subset of $\Spec(D)$ (for some integral domain $D$) and let $\mathfrak{A}$ be the set of all closures of the type $I\mapsto\bigcap\{ID_P\mid P\in\Lambda\}$, where $\Lambda$ ranges among the \emph{finite} subsets of $\Delta$. Then, $\mathfrak{A}$ is a down-directed set of finite-type closure operations (and thus correspond to an up-directed family $\mathscr{F}\subseteq\mathscr{C}$) but the infimum $s_\Delta:I\mapsto\bigcap\{ID_P\mid P\in\Delta\}$ is not of finite type (by \cite[Corollary 4.4]{FiSp}; see also Section \ref{sect:semistar} below).
\end{remark}

\section{Spaces of modules and ideals}\label{sect:submodules}
Let $R$ be any commutative ring, and let $M$ be an $R$-module. The set $\inssubmod_R(M):=\inssubmod(M)$ of the $R$-submodules of $M$ is an algebraic lattice of sets; its associated closure operation is the one sending $F\subseteq M$ to the submodule generated by $F$. In particular, the finitely generated elements of $\inssubmod(M)$ are exactly the finitely generated submodules.

The Zariski topology on $\inssubmod(M)$ is, therefore, the topology generated by the sets
\begin{equation*}
\B(x_1,\ldots,x_n):=\{N\in\inssubmod(M)\mid x_1,\ldots,x_n\in N\},
\end{equation*}
as $x_1,\ldots,x_n$ range in $M$. All the results of Section \ref{sect:alglatt} apply to $\inssubmod(M)$: thus, $\inssubmod(M)$ is a spectral space, it is locally with maximum, and the order induced by the Zariski topology is the reverse inclusion. 

\begin{Prop}\label{prop:fg-dense}
Let $M$ be an $R$-module.
\begin{enumerate}[\rm (1)]
    \item The set $\inssubfg(M)$ of finitely generated submodules of $M$ is dense in $\inssubmod(M)$, with respect to the constructible topology.
    \item $\inssubfg(M)$ is spectral if and only if $M$ is a Noetherian $R$-module.
\end{enumerate}
\end{Prop}
\begin{proof}
The two claims are the translation of points \ref{prop:propL:dense} and \ref{prop:propL:noeth} of Proposition \ref{prop:propL} to this setting.
%
%For the second one, note first that if $M$ is a Noetherian module, then $\inssubfg(M)=\inssubmod(M)$ and thus $\inssubfg(M)$ is spectral.
%
%Conversely, suppose $\inssubfg(M)$ is spectral. The set $\inssubfg(M)$ is retrocompact, since the minimum of every subbasic open set $\B(x_1,\ldots,x_n)$ belongs to $\inssubfg(M)$ and thus every $\B(x_1,\ldots,x_n)\cap\inssubfg(M)$ is quasi-compact. Therefore, by Proposition \ref{pro-retro-spec}, $\inssubfg(M)$ is proconstructible in $\inssubmod(M)$; however, by the previous proposition $\inssubfg(M)$ is dense in $\inssubmod(M)^\cons$, and thus we must have $\inssubfg(M)=\inssubmod(M)$. By definition, $M$ must be a Noetherian $R$-module, as claimed.
\end{proof}

Given a ring $R$, we denote by $\insid(R):=\inssubmod_R(R)$ the set of ideals of a ring $R$, and by $\insid^\bullet(R)$ the set of proper ideals of $R$. Endowed with the Zariski topology, both are spectral spaces, and $\insid^\bullet(R)$ is also proconstructible in $\insid(R)$, since it is the complement of the basic (quasi-compact) open set $\B(1)$. Note that the topology induced on $\Spec(R)$ by the Zariski topology on $\insid(R)$ is \emph{not} the Zariski topology of the spectrum, but rather its inverse topology. 

If $D$ is a domain, we denote by $\mathcal{F}(D)$ the set of $D$-submodules of its quotient field; if $c$ is a closure operation on $\mathcal{F}(D)$, we denote the set of $c$-closed modules by $\mathcal{F}(D)^c$.

%Any ring homomorphism $f:R\longrightarrow R'$ induces a map $f^\sharp:\insid(R')\longrightarrow\insid(R)$, given by $f^\sharp(I)=f^{-1}(I)$, for every $I\in\insid(M')$; it is easy to see that $f^\sharp$ is spectral when $\insid(R)$ and $\insid(R')$ are endowed with the Zariski topology. Therefore, $f^\sharp$ is continuous and closed when $\insid(R)$ and $\insid(R')$ are endowed with the constructible topology.

An immediate consequence of Proposition \ref{isolated-fg} is the following.
\begin{Prop}\label{prop:maximal-isolated}
Let $M$ be a maximal ideal of $R$. Then, $M$ is isolated in $\insid(R)^\cons$ if and only if $M$ is finitely generated.
\end{Prop}

\subsection{Primary ideals}
We want to study the set of primary ideals of a ring, in particular in the Noetherian case. Given a ring $R$, we denote by $\primary_R$ the set of all primary ideals of $R$; if $P$ is a prime ideal of $R$, we denote by $\primary(P)=\primary_R(P)$ be the set of $P$-primary ideals of $R$.

If $f:R\longrightarrow R'$ is a ring homomorphism, we denote by $f^\sharp:\insid(R')\longrightarrow\insid(R)$ the map given by $f^\sharp(I):=f^{-1}(I)$. When the two spaces are endowed with the Zariski topology, $f^\sharp$ is a spectral map, and thus it is continuous and closed when $\insid(R)$ and $\insid(R')$ are endowed with the constructible topology.
\begin{Prop}\label{prop:primary-noeth}
Let $R$ be a Noetherian ring, and let $P\in\Spec(R)$. Consider $\mathcal P(P)$ as a subspace of $\mathcal I^\bullet(R)^{\rm cons}$. Then 
$$
\mathcal P(P)_\infty=\chius^{\rm cons}(\mathcal P(P))=f^\sharp(\mathcal I^\bullet (R_P)),
$$
 where $f:R\to R_P$ is the localization map. In particular, if $R$ is local with maximal ideal $\mathfrak m$, then $\primary(\mathfrak m)$ is dense in $\insid^\bullet(R)$, with respect to the constructible topology. 
\end{Prop}
\begin{proof}
Suppose first that $(R,\mathfrak{m})$ is local and that $P=\mathfrak{m}$, and let $I$ be a proper ideal of $R$. Then, $R/I$ is local with maximal ideal $\mathfrak{m}/I$ and, by the Krull Intersection Theorem (see e.g. \cite[Theorem 10.17 and Corollary 10.19]{AM}), $\bigcap_{n\geq 1}(\mathfrak{m}/I)^n=(0)$; hence, $\bigcap_{n\geq 1}(\mathfrak{m}^n+I)=I$ and thus $I\in\primary(\mathfrak{m})_\infty$. The set $\primary(P)$ is closed by finite intersections; hence, $\primary(P)_\infty\subseteq\chius^\cons(\primary(\mathfrak{m}))$ by Theorem \ref{supremum}. Therefore, $\primary(\mathfrak{m})$ is dense $\insid^\bullet(R)$, i.e., $\chius^\cons(\primary(\mathfrak{m}))=\mathcal P(\mathfrak m)_\infty=\insid^\bullet(R)=f^\sharp(\insid^\bullet(R))$ (the latter equality coming from the fact that in this case $f$ is the identity).

Now let $R$ be any Noetherian ring, let $P$ be any prime ideal of $R$ and let $f:R\to R_P$ be the localization map. By the first part of the proof, $\chius^\cons(\primary(PR_P))=\insid^\bullet(R_P)$. Since $f^\sharp$ is continuous and closed, with respect to the constructible topology, we have
$$
f^\sharp(\insid^\bullet(R_P))=\chius^\cons(f^\sharp(\primary(PR_P)))=\chius^\cons(\primary(P)),
$$
as claimed.
\end{proof}

In the non-local case, we need something more than primary ideals. The following is an extension to Proposition \ref{prop:maximal-isolated} to ideals of dimension 0.
\begin{Prop}\label{prop:primary-isolated}
Let $R$ be a Noetherian ring, and let $I$ be an ideal such that $R/I$ is zero-dimensional. If the residue field of every maximal ideal containing $I$ is finite, then $I$ is isolated in $\insid(R)^\cons$.
\end{Prop}
\begin{proof}
Let $I:=\langle x_1,\ldots,x_n\rangle$. Then, $\mathcal B(x_1,\ldots,x_n)$ is clopen in the constructible topology, and thus it is enough to show that $I$ is isolated in $\mathcal B(x_1,\ldots,x_n)$, or equivalently that the zero ideal is isolated in $\insid(R/I)^\cons$.

Let $S:=R/I$. Then, $S$ is an Artin ring with all residue fields finite; hence, it is finite, and thus also $\insid(S)$ is finite. Since the constructible topology is Hausdorff, $\insid(S)^\cons$ is discrete, and in particular the zero ideal is isolated. The claim is proved.
\end{proof}

\begin{Prop}\label{prop:primary0}
Let $R$ be a Noetherian ring, and let $$\primary_0:=\{I\in\insid^\bullet(R)\mid \dim(R/I)=0\}.$$ Then the following properties hold.
\begin{enumerate}[\rm (1)]
\item\label{prop:primary0:dense} $\primary_0$ is dense in $\insid^\bullet(R)^\cons$.
\item\label{prop:primary0:mindense} If all residue fields of $R$ are finite, then $\primary_0$ is the smallest dense subset of $\insid^\bullet(R)^\cons$.
\end{enumerate}
\end{Prop}
Note that, if $(R,\mathfrak{m})$ is local, then $\primary_0=\primary(\mathfrak{m})$.
\begin{proof}
\ref{prop:primary0:dense} Clearly, $\primary_0$ is closed by finite intersections, and thus by Theorem \ref{supremum} it is enough to show that every ideal of $R$ is an intersection of elements of $\primary_0$. For every maximal ideal $\mathfrak{m}$, let  $\lambda_\mathfrak{m}:R\longrightarrow R_\mathfrak{m}$ be the canonical localization map and $\lambda^\sharp_\mathfrak{m}:\insid^\bullet(R_\mathfrak{m})\longrightarrow\insid^\bullet(R)$ be the corresponding map of ideal spaces. For every ideal $J$ of $R$, we have
\begin{equation*}
J=\bigcap\{\lambda^\sharp_\mathfrak{m}(JR_\mathfrak{m})\mid\mathfrak{m}\in\Max(R)\}
\end{equation*}
(where $JR_{\mathfrak m}$ denotes the ideal of $R_{\mathfrak m}$ generated by $\lambda_{\mathfrak m}(J)$). 
By Proposition \ref{prop:primary-noeth}, for every $\f m\in \Max(R)$ we have $\lambda_{\mathfrak m}^\sharp(JR_{\mathfrak m})\in \mathcal P(\f m)_\infty$. Since $\mathcal P(\f m)\subseteq \mathcal P_0$, we have $J\in(\primary_0)_\infty$. The first statement thus follows from Theorem \ref{supremum}.

\ref{prop:primary0:mindense} By Proposition \ref{prop:primary-isolated}, every element of $\primary_0$ is isolated in $\insid^\bullet(R)^\cons$. The claim follows.
\end{proof}

The non-Noetherian case is more complicated, since we don't have the Krull Intersection Theorem at our disposal. Indeed, even in the local case primary ideals need not to be dense. Recall that a prime ideal $\mathfrak p$ of a ring $R$ is \emph{branched} if there exists a $\mathfrak p$-primary ideal of $R$ distinct from $\mathfrak p$. A prime ideal that is not branched is called \emph{unbranched}.
\begin{Prop}\label{prop:primary-val}
Let $V$ be a valuation domain. Then, the following hold.
\begin{enumerate}[\rm (1)]
\item\label{prop:primary-val:P} If $P$ is a branched prime ideal of $V$, the closure of $\primary(P)$ in the constructible topology is equal to $\primary(P)$ plus the prime ideal directly below $P$.
\item\label{prop:primary-val:R} $\primary_V$ is a proconstructible subset of $\insid(V)$.
\end{enumerate}
\end{Prop}
\begin{proof}
\ref{prop:primary-val:P} By \cite[Theorem 17.3(3)]{gilmer}, there is a prime ideal $Q\subsetneq P$ of $V$ such that there are no prime ideals properly between $P$ and $Q$. Then, $Q$ is the intersection of all the $P$-primary ideals, and thus by Theorem \ref{supremum} is contained in the closure of $\primary(P)$ in the constructible topology. 

Let $f:V\longrightarrow V_P/QV_P$ be the natural map, and let $f^\sharp:\insid(V_P/QV_P)\longrightarrow\insid(V)$ be the induced map. Then, $f^\sharp(\insid^\bullet(V_P/QV_P))=\primary(P)\cup\{Q\}$; since $\insid^\bullet(R)$ is proconstructible in $\insid(R)$ for every ring $R$ and $f^\sharp$ is closed with respect to the constructible topology, it follows that $\primary(P)\cup\{Q\}$ is proconstructible. Considering the previous paragraph, $\primary(P)\cup\{Q\}$ must be the closure of $\primary(P)$, as claimed.

\ref{prop:primary-val:R} By Proposition \ref{prop:chiuscons-linord}\ref{prop:chiuscons-linord:cup}, we need to show that the infimum and the supremum of every nonempty subset $\Delta\subseteq\primary_V$ are in $\primary_V$.

Let thus $\Delta=\{Q_\alpha\}_{\alpha_\in A}$ and, for every $\alpha$, let $P_\alpha$ be the radical of $Q_\alpha$; let $\Delta':=\{P_\alpha\}_{\alpha\in A}$. If $\Delta'$ has a maximum, say $\overline{P}$, then the supremum of $\Delta$ is equal to the supremum of $\overline{\Delta}:=\{Q\in\Delta\mid \rad(Q)=\overline{P}\}\subseteq\primary(\overline{P})$. This set is proconstructible (if $\overline{P}$ is branched by the previous part of the proof, if $\overline{P}$ is not branched because in that case $\primary(\overline{P})=\{\overline{P}\}$), and thus it has a supremum in $\primary(\overline{P})\subseteq\primary_V$; hence, $\Delta$ has a supremum.

If $\Delta'$ has not a maximum, then for every $\alpha$ there is an $\alpha'$ such that $P_\alpha=\rad(Q_\alpha)\subsetneq\rad(Q_{\alpha'})=P_{\alpha'}$, and thus $P_\alpha\subseteq Q_{\alpha'}$; hence, the supremum $Q:=\bigcup_\alpha Q_\alpha$ of $\Delta$ is also equal to the supremum $\bigcup_\alpha P_\alpha$ of $\Delta'$. However, $\Delta'$ is contained in $\Spec(V)$, and the latter is closed in the constructible topology; hence, $Q\in\Spec(V)\subseteq\primary_V$. Therefore,  the supremum of $\Delta$ belongs to $\primary_V$.

The claim for the infimum follows similarly: if $\Delta'$ has a minimum $\overline{P}$, then $\overline{\Delta}$ has an infima which is either in $\primary(\overline{P})$ or is the prime ideal directly below $\overline{P}$, while otherwise the infimum of $\Delta'$ is the intersection of all the elements of $\Delta'$ (which is in $\Spec(V)$ and thus in $\primary_V$). In both cases, the infimum is in $\primary_V$. Therefore, $\primary_V$ is closed by suprema and infima and so it is proconstructible.
\end{proof}

Note that, if $\dim V>1$, then $\primary_V\neq\insid^\bullet(V)$: for example, if $\mathfrak{m}$ is the maximal ideal of $V$ and $x$ belongs to a prime ideal $P$ strictly contained in $\mathfrak{m}$, then $x\mathfrak{m}$ is not primary: indeed, if $y\in\mathfrak{m}\setminus P$, then $xy\in x\mathfrak{m}$, while $x\notin x\mathfrak{m}$ and $y^n\notin P$ and thus $y^n\notin x\mathfrak{m}$ for every $n\geq 1$.

\subsection{Principal ideals}
Proposition \ref{prop:propL}\ref{prop:propL:dense} shows that finitely generated substructures form a dense subset (with respect to the constructible topology) of the set of all substructures. In this section, we show that the same does not hold if we replace ``finitely generated'' with ``cyclic'' (or ``principal'').

Given a ring $R$, let $\insprinc(R)$ denote the set of principal ideals of $R$. We first study the Noetherian case.

%\begin{Prop}\label{prop:maximal-isolated}
%Let $M$ be a $R$-module, and let $N\in\inssubmod(M)$ be maximal among the proper submodules of $M$. Then, $N$ is an isolated point of $\inssubmod(M)^\cons$ if and only if it is finitely generated.
%\end{Prop}
%\begin{proof}
%Apply Proposition \ref{isolated-fg}.
%\end{proof}

\begin{Prop}\label{prop:princ-noeth-PIR}
Let $R$ be a Noetherian ring. Then, $\insprinc(R)$ is dense in $\insid(R)^\cons$ if and only if $R$ is a principal ideal ring.
\end{Prop}
\begin{proof}
If $R$ is a principal ideal ring, then every ideal is principal and thus $\insprinc(R)$ is the whole $\insid(R)$.

Conversely, suppose that $\insprinc(R)$ is dense in $\insid(R)^\cons$. Since $R$ is Noetherian, by Proposition \ref{prop:maximal-isolated} every maximal ideal is an isolated point of $\insid(R)^\cons$; therefore, $\insprinc(R)$ must contain every maximal ideal. By \cite[Theorem 12.3]{kaplansky-edm}, $R$ is a principal ideal ring, as claimed.
\end{proof}

Proposition \ref{prop:princ-noeth-PIR} doesn't work in the non-Noetherian setting; for example, if $D$ is a B\'ezout domain then $\insprinc(D)=\insid_f(D)$ is dense in the constructible topology, but $D$ may not be Noetherian. However, we can say something; recall that the $v$-operation is the \emph{divisorial closure} $v:I\mapsto(D:_K(D:_KI))$ (where $K$ is the quotient field of $D$) while the $t$-operation is the finite-type closure associated to $v$, i.e., $t=v_f$. An ideal is \emph{divisorial} if it is $v$-closed.
\begin{Prop}\label{prop:dense-top}
Let $D$ be an integral domain. If $\insprinc(D)$ is dense in $\insid(D)^\cons$, then the $t$-operation of $D$ is equal to the identity.
\end{Prop}
\begin{proof}
Since every principal ideal is divisorial, we have $\insprinc(R)\subseteq\mathcal{F}(D)^v\subseteq\mathcal{F}(D)^t$. By Proposition \ref{prop:Fstar} (or \cite[Proposition 2.9 and Corollary 2.10]{fontana-loper-nagata}), $\mathcal{F}(D)^t$ is proconstructible in $\mathcal F (D)$; it follows that $\mathcal{F}(D)^t\cap \mathcal I(D)$ is proconstructible in $\mathcal I(D)$. Hence, since $\insprinc(D)$ is dense in $\mathcal I(D)^{\rm cons}$, it must be $\mathcal I(D)\subseteq \mathcal{F}(D)^t$. It follows that the $t$-operation is equal to the identity (since two semistar operations of finite type are equal if and only if they agree on integral ideals).
\end{proof}

\begin{corollary}
Let $D$ be an integrally closed integral domain. If $\insprinc(D)$ is dense in $\insid(D)^\cons$, then $D$ is a Pr\"ufer domain.
\end{corollary}
\begin{proof}
By Proposition \ref{prop:dense-top}, if $\insprinc(D)$ is dense then the $t$-operation is the identity. Since $D$ is integrally closed, it follows that $D$ is a Pr\"ufer domain, in view of \cite[Proposition 34.12]{gilmer}.
\end{proof}

Note that this corollary cannot be inverted -- for example, a Dedekind domain that is not a PID gives an example of a Pr\"ufer domain whose principal ideals are not dense, by Proposition \ref{prop:princ-noeth-PIR}. We advance the following

\textbf{Conjecture}: if $D$ is integrally closed and $\insprinc(D)$ is dense in $\insid(D)^\cons$, then $D$ is a B\'ezout domain.

\medskip

We can also ask the opposite question: when is $\insprinc(D)$ proconstructible in $\insid(D)$?
\begin{Prop}\label{prop:accp}
Let $R$ be a ring. If $\insprinc(R)$ proconstructible in $\insid(R)$, then $R$ satisfies the ascending chain condition on principal ideals.
\end{Prop}
\begin{proof}
Suppose that there is an ascending chain $\Delta:=\{(r_\alpha)\mid \alpha\in A\}$ of principal ideals that does not stabilize. As a set, $\Delta$ is closed by finite suprema (since it is a chain), and thus since $\insprinc(R)$ is closed in the constructible topology then $\Delta_\infty\subseteq\insprinc(R)$, by Theorem \ref{supremum}. In particular the union $\bigcup_\alpha(r_\alpha)$ must be a principal ideal, say generated by $r$. However, this implies $r\in(r_{\overline{\alpha}})$ for some $\overline{\alpha}\in A$, which in turn implies that $\Delta$ stabilizes at $(r_{\overline{\alpha}})$. This is a contradiction, and thus the conclusion follows.
\end{proof}

In particular, the previous proposition rules out every Pr\"ufer domain that is not Noetherian. On the positive side, by \cite[Theorem 7.9.5]{found-comm-ring-mod}, a domain $D$ is a UFD if and only if $\insprinc(D)=\mathcal{F}(D)^t\cap\insid(D)$, and thus if $D$ is a UFD then $\insprinc(D)$ is proconstructible. We give two extensions of this result; the first one uses a method similar to the proof of Proposition \ref{prop:primary-isolated}.
\begin{theorem}\label{thm:Krull}
Let $D$ be a Krull domain. Then, $\insprinc(D)$ is proconstructible in $\insid(D)$.
\end{theorem}
\begin{proof}
Both $\mathcal{F}(D)^t$ and $\insid(D)$ are proconstructible in $\mathcal{F}(D)$; hence, the constructible topology of $X:=\mathcal{F}(D)^t\cap\insid(D)$ is the restriction of the constructible topology of $\mathcal{F}(D)$ (or of $\mathcal I(D)$). Since $\insprinc(D)\subseteq X$, it is enough to show that $\insprinc(D)$ is proconstructible in the spectral space $X$; to this end, we prove the following claim: \emph{each nonzero element of $X$ is isolated in $X^\cons$}.

Indeed, let $N\in X$ be nonzero. Then, there is a finite subset $F\subset N$ such that $N=(F)^t$ \cite[Corollary 44.3]{gilmer}; therefore, $\{N\}^\uparrow=\B(F)\cap X$ is a clopen set of the constructible topology, and thus it is enough to show that $N$ is isolated in ${\{N\}^\uparrow}$, with respect to the constructible topology. Let $Y$ be the set of height-one prime ideals of $D$.

If $N=D$ then $\{N\}^\uparrow=\{N\}$ and we are done. Otherwise, by \cite[Corollaries 43.9 and 44.8]{gilmer}, $N$ is contained in finitely many height-one primes, say $P_1,\ldots,P_k$. The set $X$ is a semigroup under the \emph{$t$-product} $I\times_tJ:=(IJ)^t$ (see \cite{Griffin-can-j-math}); by \cite[Chapter VII, \textsection 1, Theorem 2]{bourbaki}, $X$ is isomorphic to the free semigroup generated by $Y$, and $I$ divides $J$ in $X$ if and only if $J\subseteq I$. Therefore, $\{N\}^\uparrow\setminus\{N\}$ has only finitely many minimal elements, namely $(N(D:P_1))^t,\ldots,(N(D:P_k))^t$. Like for the case of $N$, for each $i$ there is a finite set $G_i\subseteq D$ such that $(G_i)^t=(N(D:P_i))^t$; thus, each set $\{(N(D:P_i))^t\}^\uparrow$ ($1\leq i\leq k$) is clopen in $X^\cons$. Thus, their finite union $\{N\}^\uparrow\setminus\{N\}$ is also clopen; it follows that $N$ is isolated in ${\{N\}^\uparrow}$, with respect to the constructible topology. The claim  is proved. 

In particular, each nonprincipal ideal of $X$ is isolated in $X$, and thus $X\setminus\insprinc(D)$ is open in $X^\cons$; it follows that $\insprinc(D)$ is proconstructible in $X$. The conclusion is now clear.
\end{proof}

\begin{Prop}\label{prop:gcd}
Let $D$ be a GCD domain, and let $\mathcal{F}_p(D)$ be the set of cyclic submodules of the quotient field of $D$. Then:
\begin{enumerate}[\rm (1)]
\item\label{prop:gcd:Fp} $\chius^\cons(\mathcal{F}_p(D))=\mathcal{F}(D)^t$;
\item\label{prop:gcd:princ} $\chius^\cons(\insprinc(D))=\mathcal{F}(D)^t\cap\insid(D)$;
\item\label{prop:gcd:chiuso} The following conditions are equivalent.
\begin{enumerate}[\rm (i)]
    \item\label{prop:gcd:chiuso:proc} $\insprinc(D)$ is proconstructible in $\mathcal I(D)$.
    \item\label{prop:gcd:chiuso:UFD} $D$ is a unique factorization domain. 
\end{enumerate}
\end{enumerate}
\end{Prop}
\begin{proof}
Since $D$ is a GCD domain, the intersection of two principal ideals (and thus of two principal fractional ideals) is again principal \cite[Theorem 16.2]{gilmer}; by Theorem \ref{supremum},  $\mathcal F_p(D)_\infty\subseteq\chius^{\rm cons}(\mathcal{F}_p(D))$. By \cite[Proposition 5.2.2(10)]{found-comm-ring-mod}, $\mathcal F_p(D)_\infty$ is the set $\mathcal F (D)^v\setminus\{K\}$ of proper $D$-submodules of $K$ that are $v$-closed (where $v$ is the divisorial closure). 
    
By Proposition \ref{prop:Fstar}, $\chius^\cons(\mathcal{F}(D)^v)=\mathcal{F}(D)^{v_f}=:\mathcal{F}(D)^t$, and thus it easily follows that $\chius^\cons(\mathcal{F}_p(D))=\mathcal{F}(D)^t$. Thus \ref{prop:gcd:Fp} is proved.

The equality \ref{prop:gcd:princ} follows in exactly the same way, because a divisorial integral ideal is also the intersection of principal integral ideals of $D$ (since $\alpha D\cap D$ is a principal integral ideal for every $\alpha\in K$).

\ref{prop:gcd:chiuso}. By Theorem \ref{thm:Krull} (or the discussion before it), it suffices to show \ref{prop:gcd:chiuso:proc}$\Longrightarrow$\ref{prop:gcd:chiuso:UFD}. This follows from Proposition \ref{prop:accp} and \cite[Theorem 5.1.20]{found-comm-ring-mod}.
%The last point now follows from the fact that a GCD domain is a UFD if and only if all its integral (or fractional) $t$-ideals are principal.
\end{proof}

\section{Overrings of an integral domain}\label{sect:overrings}
Let $A\subseteq B$ be a ring extension and let $R(B|A)$ denote the set of all subrings of $B$ containing $A$. Then, $R(B|A)$ is another example of algebraic lattice of sets: the associated closure operation sends $F\subseteq B$ to the ring $A[F]$ generated by $A$ and $F$. In particular, $R(B|A)$ is a spectral space, it is locally with maximum, and the set of finitely generated elements of $R(B|A)$ is the set of finitely generated $A$-algebras contained in $B$. In particular, the set of finitely generated $A$-algebras is dense in $R(B|A)^\cons$ (this had been observed, in the case of domain, in \cite[Proposition 7.6]{ZarNoeth}).

Let now $D$ be an integral domain and $K$ the quotient field of $D$. We denote by:
\begin{itemize}
\item $\Over(D):=R(K|D)$ the set of all \emph{overrings} of $D$;
\item $\Zar(D)$ the set of all valuation overrings of $D$;
\item $\overic(D)$ the set of integrally closed overrings of $D$;
\item $\pruf(D)$ the set of Pr\"ufer overrings of $D$;
\item $\prufsloc(D)$ the set of semilocal Pr\"ufer overrings of $D$.
\end{itemize}
By \cite[Example 2.1(3), Propositions 2.12, 3.2 and 3.6]{Fi}, we see that that $\Zar(D)$ and $\overic(D)$ are proconstructible in $\Over(D)$; note that the latter is again an algebraic lattice of sets, while the former is not (since the intersection of two noncomparable valuation domains is not a valuation domain). Our next results show that Pr\"ufer overrings usually do not give rise to spectral spaces.

\begin{comment}
Let us endow $R(B|A)$ with the Zariski topology, i.e., the topology whose basic open sets are those of the form
$$
\mathcal B(F):=\{R\in R(B|A):F\subseteq R \}
$$
where $F$ runs among finite subsets of $B$. By \cite[Proposition 3.5]{Fi}, $R(B|A)$ is a spectral space. 
\begin{remark}
Let $A\subseteq B$ be a ring extension. 
\begin{enumerate}
	\item Clearly, for every $R,R'\in  R(B|A)$, we have $R\in \chius(\{R'\})$ if and only if $R\subseteq R'$. Thus the order induced by the Zariski topology on $R(B|A)$ is the reverse inclusion. 
	\item $R(B|A)$ is locally with maximum. Indeed, every basic open set $\mathcal B(F)$ (where $F$ is a finite subset of $B$) has the minimum under inclusion, namely $A[F]$, and thus, by part (1), it is the maximum with respect to the order induced by the Zariski topology. 
\end{enumerate}
\end{remark}
If $D$ is an integral domain and $K$ is  the quotient field of $D$, we set $\Over(D):=R(K|D)$, i.e., $\Over(D)$ is the space of all overrings of $D$. Furthermore, we denote by $\Zar(D)$ the subspace of $\Over(D)$ consisting of all valuation overrings of $D$. By keeping in mind \cite[Example 2.1(3), Propositions 2.12 and 3.2]{Fi} it is immediately seen that $\Zar(D)$ is proconstructible in $\Over(D)$. 
\end{comment}
\begin{Prop}\label{semilocal-prufer-dense}
Let $D$ be an integral domain. Then, the closure of $\prufsloc(D)$ in the constructible topology of $\Over(D)$ is $\overic(D)$.
\end{Prop}
\begin{proof}
Consider the proconstructible subset $Y:=\Zar(D)$ of the spectral space $X:=\Over(D)$. By \cite[Theorem 22.8]{gilmer}, $Y_f={\prufsloc}(D)$; on the other hand, by \cite[Corollary 5.22]{AM}, $Y_\infty=\overic(D)$. Since $\Over(D)$ is locally with maximum, the conclusion follows from Proposition \ref{density-locally-with-maximum}.
\end{proof}

\begin{corollary}\label{cor:pruf-comp}
For an integral domain $D$, the following conditions are equivalent. 
\begin{enumerate}[\rm (i)]
	\item\label{cor:pruf-comp:pruf} The integral closure $\overline{D}$ of $D$ is a Pr\"ufer domain. 
	\item\label{cor:pruf-comp:qc} $\pruf(D)$ is quasi-compact as a subspace of $\Over(D)$. 
\end{enumerate}
\end{corollary}
\begin{proof}
\ref{cor:pruf-comp:pruf}$\Longrightarrow$\ref{cor:pruf-comp:qc} is trivial since, if $\overline{D}$ is a Pr\"ufer domain, then $\pruf(D)=\overic(D)$ by \cite[Corollary 4.5]{Bu-Va}. 

\ref{cor:pruf-comp:qc}$\Longrightarrow$\ref{cor:pruf-comp:pruf}. Since the order of the Zariski topology on $\Over(D)$ is the reverse inclusion, $\pruf(D)$ is closed under generizations, again by \cite[Corollary 4.5]{Bu-Va}. Since by assumption $\pruf(D)$ is quasi-compact, then it is closed in the inverse topology and thus proconstructible, by \cite[Proposition 2.6]{fifolo2}. The inclusions $\prufsloc(D)\subseteq \pruf(D)\subseteq \overic(D)$ and Proposition \ref{semilocal-prufer-dense} imply $\pruf(D)=\overic(D)$, and in particular $\overline{D}\in\pruf(D)$, i.e., $\overline{D}$ is a Pr\"ufer domain.
\end{proof}

We now want to prove a ``Noetherian'' analogue of Proposition \ref{semilocal-prufer-dense}. We start by considering discrete valuation rings; for the notation $\wedge_\Delta$ and the $b$-operation see the following Section \ref{sect:semistar}.
\begin{theorem}\label{DVR-dense}
Let $D$ be a Noetherian domain and let $\Delta(D)$ be the set of all discrete valuation overrings of $D$. Then $\Delta(D)$ is dense in $\Zar(D)^\cons$.
\end{theorem}
\begin{proof}
By \cite[Proposition 6.8]{swanson-huneke}, for every finitely generated ideal $I$ of $R$ we have $I^{\wedge_{\Delta(D)}}=I^b$, where $\wedge_{\Delta(D)}$ is the semistar operation induced by $\Delta(D)$ and $b$ is the $b$-operation (or integral closure) on $D$. By \cite[Lemma 5.8(3)]{spectral-semistar}, it follows that $\Zar(D)=\chius^\inverse(\Delta(D))$, i.e., that $\Delta(D)$ is dense in $\Zar(D)$ with respect to the inverse topology.

For every finite subset $F$ of the quotient field $K$ of $D$, let $\B(F):=\Zar(D[F])$ denote (with a small abuse of notation) the generic basic open set of the Zariski topology (induced by that of $\inssubmod_D(K)$). Since $\B(F)\cap\B(G)=\B(F\cup G)$, the open and quasi-compact subspaces of $\Zar(D)$ are precisely all finite unions of basic open sets of the type $\B(F)$, where $F$ is a finite subset of $K$. Since the constructible topology is, by definition, the coarsest topology on $\Zar(D)$ for which open and quasi-compact subspaces of the Zariski topology are clopen, it is easily seen that a basis of $\Zar(R)^\cons$ consists of sets of the type 
\begin{equation*}
\B(F)\cap\left(\Zar(D)\setminus \bigcup_{j=1}^m\B(G_j)\right)=\Zar(D[F])\cap\left(\Zar(D[F])\setminus \bigcup_{j=1}^m\B(G_j)\right),
\end{equation*}
for some finite subsets $F,G_1,\ldots,G_m$ of $K$. Let $\Omega$ be the previous set. Then, $\Omega$ is an open set of the inverse topology of $\Zar(D[F])$; by the first paragraph of the proof, $\emptyset\neq\Delta(D[F])\cap\Omega\subseteq\Delta(D)\cap\Omega$. Therefore, $\Delta(D)$ intersects all basic open sets of $\Zar(D)^\cons$, and thus it is dense in it. The claim is proved. 
\end{proof}

\begin{corollary}
\label{cor:Noeth-dense}
Let $D$ be a Noetherian domain. Then the set of the overrings of $D$ that are Dedekind and semilocal is dense in $\overic(D)$, with respect to the constructible topology. 
\end{corollary}
\begin{proof}
Let $\Delta:=\Delta(D)$ as in Theorem \ref{DVR-dense} and let $\Lambda $ be the set of the overrings of $D$ that are Dedekind and semilocal. By \cite[Theorem 12.2]{Matsumura}, $\Delta_f=\Lambda$ and, by Proposition \ref{density-locally-with-maximum}\ref{density-locally-with-maximum:dense} and Theorem \ref{DVR-dense}, $\Lambda$ is dense in $\Zar(D)_\infty=\overic(D)$, as claimed.
\end{proof}

\begin{comment}
The following result is a slight generalization of \cite[Proposition 7.6]{ZarNoeth}, with a different proof. It can be seen as an analogue of Proposition \ref{prop:fg-dense} for rings instead of modules.
\begin{Prop}\label{prop:algfg}
Let $A\subseteq B$ be a ring extension, and let $Y$ be the set of all $A$-subalgebras of $B$ that are of finite type over $A$. Then, $Y$ is dense in $R(B|A)$, with respect to the constructible topology. 
% \begin{enumerate}[(a)]
% \item\label{lemma:algfg:alg} The set $\mathcal{F}(D):=\{D[x_1,\ldots,x_n]\mid x_1,\ldots,x_n\in K\}$ is dense in $\Over(D)$, with respect to the constructible topology..
% \item\label{lemma:algfg:ic} The set $\overline{\mathcal{F}}(D):=\{\overline{D[x_1,\ldots,x_n]}\mid x_1,\ldots,x_n\in K\}$ (where $\overline{T}$ indicates the integral closure of $T$) is dense in $\overic(D)$, with respect to the constructible topology..
% \end{enumerate}
\end{Prop}
\begin{proof}
Use Example \ref{algebraic-lattice-ex}, \ref{algebraic-lattice-density} and the obvious equality $Y=R(B|A)_{\rm fin}$. 
\end{proof}
\end{comment}

As a consequence of Corollary \ref{cor:Noeth-dense}, we can complete \cite[Proposition 7.3]{ZarNoeth} by considering the case of principal ideal domains.
\begin{Prop}\label{prop:PID}
Let $D$ be a Noetherian domain, and let $$\Delta:=\{T\in\Over(D)\mid T \mbox{ is a principal ideal domain}\}.$$ Then the following conditions are equivalent:
\begin{enumerate}[\rm (i)]
\item\label{prop:PID:ic} the integral closure of $D$ is a principal ideal domain;
\item\label{prop:PID:cons} $\Delta$ is proconstructible in $\Over(D)$;
\item\label{prop:PID:qc} $\Delta$ is quasi-compact.
\end{enumerate}
\end{Prop}
\begin{proof}
Note first that \ref{prop:PID:cons} and \ref{prop:PID:qc} are equivalent since, in view of \cite[Corollary 5.3]{Bu-Va},  $\Delta$ is closed under generization (i.e., if $T\subseteq T'$ and $T\in\Delta$, then also $T'\in\Delta$).

If \ref{prop:PID:ic} holds, then $\Delta$ is equal to $I(D)$, which is proconstructible by Proposition \ref{semilocal-prufer-dense}, and so \ref{prop:PID:cons} holds. Suppose \ref{prop:PID:cons} holds: by Corollary \ref{cor:Noeth-dense}, the set $\Lambda$ of overrings of $D$ that are Dedekind and semilocal is dense in $I(D)$; since $\Lambda\subseteq\Delta$ (see, for instance, \cite[Corollary 34.7]{gilmer}), also $\Delta$ is dense in $I(D)$. Since $\Delta$ is proconstructible, we must have $\Delta=I(D)$, and in particular the integral closure $\overline{D}$ of $D$ is in $\Delta$, as claimed.
\end{proof}

An integral domain $D$ is called \emph{rad-colon coherent} if the radical of the conductor $(D:_Dx)$ is the radical of a finitely generated ideal for every $x\in K$ (where $K$ is the quotient field of $D$); likewise, it is called \emph{rad-colon principal} if the radical of each $(D:_Dx)$ is the radical of a principal ideal for every $x\in K$ \cite{localizzazioni}.
\begin{Prop}\label{prop:rcc}
Let $D$ be a rad-colon coherent domain. If $\{P_\alpha\}_{\alpha\in A}$ is a chain of prime ideals of $D$ and $P:=\bigcup_\alpha P_\alpha$, then $\bigcap_\alpha D_{P_\alpha}=D_P$.
\end{Prop}
\begin{proof}
Since $D$ is rad-colon coherent, the set $X$ of localizations of $D$ at prime ideals is proconstructible in $\Over(D)$ \cite[Theorem 3.2(b)]{localizzazioni}; in particular, the constructible closure of $\Delta:=\{D_{P_\alpha}\}_{\alpha\in A}$ is contained in $X$. In the Zariski topology, $\sup\Delta$ is exactly the intersection of the elements of $\Delta$; by Theorem \ref{supremum}, it follows that $\sup\Delta\in X$. It follows that $\bigcap_\alpha D_{P_\alpha}=D_P$, as claimed.
\end{proof}

\begin{Prop}
Let $D$ be a rad-colon principal domain. If $\{S_\alpha\}_{\alpha\in A}$ is a descending chain of multiplicatively closed subsets of $D$, and $S:=\bigcap_\alpha S_\alpha$, then $\bigcap_\alpha S_\alpha^{-1}D=S^{-1}D$.
\end{Prop}
\begin{proof}
The proof is the same as the previous proposition, using \cite[Theorem 4.4]{localizzazioni} to prove that the set of quotient rings of $D$ is proconstructible.
\end{proof}

We show that the hypothesis in the previous two propositions cannot be dropped in general.

\begin{example}~
\begin{enumerate}
\item Let $V$ be a valuation domain with an unbranched maximal ideal $M$ and such that the residue field is equal to $K(X)$ for some field $K$ and some indeterminate $X$ over $K$. Let $D$ be the pullback of $K$ in $V$: that is, $D:=\{r\in V\mid \pi(r)\in K\}$, where $\pi:V\longrightarrow V/M$ is the residue map. By \cite[Theorem 1.4]{Fo-1980}, the prime spectrum of $D$ is (set-theoretically) equal to the spectrum of $V$; in particular, $M$ is the maximal ideal of $D$. %Furthermore, $D$ is not rad-colon coherent since if $\pi(x)=X$ then the open set  of $\Spec(D)$ consisting of all prime ideals of $D$ not containing $(D:_Dx)$ is $\Spec(D)\setminus\{M\}=\Spec(V)\setminus\{M\}$, and it is not quasi-compact, by Lemma \ref{max-compact} and \cite[Theorem 17.3(e)]{gilmer} (as $M$ is not branched). Thus the radical of $(D:_Dx)$ is not the radical of any finitely generated ideal. 

Let $\{P_\alpha\}$ be the chain of non-maximal prime ideals of $D$: then, for every $\alpha$, we have $D_{P_\alpha}=V_{P_\alpha}$, and thus $\bigcap_\alpha D_{P_\alpha}=V$. On the other hand, $M=\bigcup_\alpha D_{P_\alpha}$, and $D_M=D$; thus, $\bigcap_\alpha D_{P_\alpha}\neq D_M$. In particular, $D$ is not rad-colon coherent.

\item Let $D$ be a Dedekind domain with countably many maximal ideals, say $M_0,M_1,\ldots,M_n,\ldots$ and $M_\infty$, and suppose that $M_\infty\subseteq\bigcup_nM_n$: that is, suppose that $M_\infty$ is not the radical of any principal ideal, or equivalently that the class of $M_\infty$ in the Picard group of $D$ is not torsion. In particular, $D$ is not rad-colon principal.

Let $S_n:=D\setminus(M_1\cup\cdots\cup M_n)$: then, $\{S_n\}_{n\in\mathbb{N}}$ is a descending chain of multiplicatively closed subset whose intersection $S=D\setminus\bigcup_nM_n$ is just the set of units of $D$. Hence, $S_n^{-1}D=D_{M_1}\cap\cdots\cap D_{M_n}$, and thus $T:=\bigcap_nS_n^{-1}D=\bigcap_nD_{M_n}$. Since $D$ is a Dedekind domain, the maximal ideals of $T$ are in the form $M_nT$, for $n\in\mathbb{N}$; in particular, $M_\infty T=T$. On the other hand, $S^{-1}D=D$; in particular, $\bigcap_nS_n^{-1}D\neq S^{-1}D$.
\end{enumerate}
\end{example}

\subsection{Spaces of semistar operations.}\label{sect:semistar}
Let $D$ be an integral domain with quotient field $K$. As before, let $\mathcal F (D):=\inssubmod_D(K)$ denote the set of all $D$-submodules of $K$. A \emph{semistar operation} on $D$ is a closure operation $\star$ on $\mathcal{F}(D)$ such that $(kF)^\star=kF^\star$ for all $k\in K$ and all $F\in\mathcal{F}(D)$. As $\mathcal{F}(D)$ is an algebraic lattice of sets, keeping in mind Remark \ref{oss:closL} all the definition and results of Section \ref{sect:alglatt} apply to semistar operations. 

We denote by $\inssemistar(D)$ and $\inssemistartf(D)$, respectively, the set of all semistar operations and the set of all semistar operations of finite type. For any $\star\in \inssemistar(D)$, let $\mathcal F (D)^\star:=\{I\in \mathcal F (D)\mid I=I^\star\}$.  By Corollary \ref{cor:closL}, $\inssemistartf(D)$ is proconstructible in $\insclos_f(K)$, and thus it is a spectral space when endowed with the restriction of the Zariski topology of $\insclos_f(K)$. Since $x\in I^\star$ if and only if $1\in(x^{-1}I)^\star$, it follows that it is enough to consider the sets $V(I,x)$ with $x=1$, i.e., a subbasis for the Zariski topology on $\inssemistartf(D)$ is formed by the sets
\begin{equation*}
V_F:=V(F,1)=\{\star\in \SStar_f(D)\mid 1\in F^\star \},
\end{equation*}
as $F$ ranges among the $D$-submodules of $K$. (The spectrality of this set was also proved through ultrafilter techniques in \cite[Theorem 2.13]{FiSp}.) We can also consider this topology on the whole $\inssemistar(D)$, but it this case we may not obtain a spectral space \cite[Section 4]{non-ft}. 

For every overring $R$ of $D$, the map $\star_R$ defined by setting $F^{\star_R}:=FR$, for every $F\in \mathcal F (D)$, is a semistar operation of finite type. if $\Delta\subseteq\Over(D)$, then we set $\wedge_\Delta:=\bigwedge(\{\star_R\mid R\in \Delta\})$, and in particular we set $b:=\wedge_{\Zar(D)}$.

There is a natural topological embedding $\iota:\Over(D)\longrightarrow\inssemistartf(D)$, defined by $\iota(R):=\star_R$. (This is the reason why the topology is called the Zariski topology.) We will use our results to show that this map is, in general, not spectral.

\begin{Prop}\label{prop:controes-semistar}
Let $D$ be an integral domain, and let $\iota:\Over(D)\longrightarrow\inssemistartf(D)$ the canonical embedding. If $\iota$ is a spectral map (equivalently, if $\iota(\Over(D))$ is closed in the constructible topology of $\inssemistartf(D)$) then the integral closure $\overline{D}$ of $D$ is a Pr\"ufer domain.
\end{Prop}
\begin{proof}
Let $\Delta$ be the set of semilocal Pr\"ufer overrings of $D$.
We first prove that if $T\in\Delta$, then $\iota(T):=\star_T$ is the infimum of a finite family of elements of $\iota(\Zar(D))$: indeed, if $I$ is any $D$-submodule of the quotient field $K$ of $D$ and $V_1,\ldots,V_n$ are the localizations at the maximal ideals of $T$,  then, keeping in mind \cite[Theorem 4.10]{gilmer}, we have 
\begin{equation*}
I^{\star_T}=IT=IT\left(\bigcap_{i=1}^nV_i\right)=\bigcap_{i=1}^nITV_i=\bigcap_{i=1}^nI^{\star_{V_i}},
\end{equation*}
i.e., $\star_T=\inf_{1\leq i\leq n}\star_{V_i}$. On the other hand, if $V_1,\ldots,V_n\in\Zar(D)$, then $T:=V_1\cap\cdots\cap V_n$ is a Pr\"ufer domain whose localizations at the maximal ideals are a subset of $\{V_1,\ldots,V_n\}$, and so $\inf_{1\leq i\leq n}\star_{V_i}=\star_T$. Therefore, the set $\iota(\Delta)$ is closed by finite infima (that is, it is closed by finite suprema, with respect to the Zariski topology).

Suppose $\iota(\Over(D))$ is proconstructible in $\inssemistartf(D)$. By Theorem \ref{supremum}, it follows that $$\iota(\Delta)_\infty\subseteq \iota(\Over(D)),
$$
and thus, in particular, $b=\wedge_{\Zar(D)}\in \iota(\Over(D))$. Let $R\in \Over(D)$ be such that $b=\iota(R)=\star_R$. Then, $D^b=D^{\star_R}=DR=R$; however, $D^b$ is equal to the integral closure $\overline{D}$ of $D$ \cite[Corollary 5.22]{AM}, and thus $R=\overline{D}$. Hence, for every $\overline{D}$-submodule $J$ of $K$, we have $J^b=J^{\star_{\overline{D}}}=J\overline{D}=J$. This proves that $b$ is the identity on $\mathcal F(D)$, and this happens if and only if $\overline{D}$ is a Pr\"ufer domain, by \cite[Theorem 24.7]{gilmer}. The conclusion follows. 
\end{proof}

By \cite[Proposition 2.7]{FiSp}, if $\mathcal{S}$ is a quasi-compact set of semistar operations then $\bigwedge(\mathcal{S})$ is a semistar operation of finite type; the converse holds if $\mathcal{S}$ is a set of closures in the form $\star_R$ if the rings $R$ are either all localizations of $D$ \cite[Corollary 4.4]{FiSp} or valuation overrings of $D$ \cite[Proposition 4.5]{FiSp}. In \cite{FiSp}, it was also conjectured that the converse is valid for every family of semistar operations induced by overrings; while this conjecture has already been disproved in \cite[Example 3.6]{graz} using numerical semigroup rings, we now show that it can fail on every non-Pr\"ufer domain.
\begin{example}
Let $D$ be an integrally closed domain that is not a Pr\"ufer domain and, as before, let $b$ be the semistar operation $b:I\mapsto\bigcap\{IV\mid V\in\Zar(D)\}$, for every $I\in \mathcal{F}(D)$. Note that $b$ is a semistar operation of finite type (for example, since $\Zar(D)$ is quasi-compact). However, if $T$ is a Pr\"ufer overring of $D$ and $I$ is any $D$-submodule of the quotient field of $D$, then, again by \cite[Theorem 4.10]{gilmer}, we have
\begin{equation*}
IT=\bigcap_{P\in\Spec(T)}IT_M=\bigcap_{V\in\Zar(T)}IV,
\end{equation*}
and thus $b=\bigwedge(\mathcal{S})$, where $\mathcal{S}:=\{\star_R\mid R\in\pruf(D)\}$. In particular, $\bigwedge(\mathcal{S})$ is of finite type, while $\mathcal{S}$ is not quasi-compact since it is homeomorphic to $\pruf(D)$ \cite[Proposition 2.5]{FiSp} and $\pruf(D)$ is not quasi-compact by Corollary \ref{cor:pruf-comp}.
\end{example}
\noindent
{\bf Acknowledments.}  The authors would like to express their sincere gratitude to the referee, whose remarkable suggestions helped to significantly improve  the paper.
\bibliographystyle{plain}
\bibliography{biblio}

\end{document}